\documentclass[11pt]{amsart}
\usepackage[margin=1in]{geometry}
\usepackage{amssymb,amsmath}
\usepackage{hyperref}

\newtheorem{theorem}{Theorem}[section]
\newtheorem{lemma}[theorem]{Lemma}
\newtheorem{proposition}[theorem]{Proposition}
\newtheorem{corollary}[theorem]{Corollary}
\newtheorem{mainthm}{Theorem}

\theoremstyle{definition}
\newtheorem{definition}[theorem]{Definition}
\newtheorem{remark}[theorem]{Remark}
\newtheorem{example}[theorem]{Example}
\newtheorem{question}[theorem]{Question}

\DeclareMathOperator{\Ann}{Ann}
\DeclareMathOperator{\Spec}{Spec}
\DeclareMathOperator{\Hom}{Hom}
\DeclareMathOperator{\Frac}{Frac}
\DeclareMathOperator{\rad}{rad}
\DeclareMathOperator{\rrad}{\mathbb{R}\text{-}rad}
\newcommand{\p}{\mathfrak p}
\newcommand{\q}{\mathfrak q}
\newcommand{\m}{\mathfrak m}
\newcommand{\M}{\mathcal M}
\newcommand{\F}{\mathcal F}
\newcommand{\G}{\mathcal G}
\newcommand{\OO}{\mathcal O}
\newcommand{\Z}{\mathbb Z}
\newcommand{\Q}{\mathbb Q}
\newcommand{\R}{\mathbb R}
\newcommand{\N}{\mathbb N}

\title[Nullstellens\"atze for modules and sheaves]{Hilbert's Nullstellensatz and the real Nullstellensatz for modules and sheaves}
\author{Jakob Cimpri\v c}
\address{University of Ljubljana, Faculty of Mathematics and Physics, Department of Mathematics, Jadranska 21, SI-1000 Ljubljana, Slovenia, and Institute of Mathematics, Physics and Mechanics, Jadranska 19, SI-1000 Ljubljana, Slovenia}
\email{jaka.cimpric@fmf.uni-lj.si}
\thanks{The author acknowledges the financial support from the Slovenian Research and Innovation Agency (ARIS), research core funding No.\ P1-0222 and research projects No.\ J1-60011 and No.\ J1-50002.}
\subjclass[2020]{Primary 13C05; Secondary 13J30, 14P05, 13A15, 14F06, 18E10}
\keywords{Nullstellensatz, real Nullstellensatz, Hilbert module, maximal submodule, prime submodule, semiprime submodule, real submodule, semireal module, Jacobson ring, coherent sheaf}

\usepackage{array}
\begin{document}

\begin{abstract}
We prove the geometric halves of Hilbert's Nullstellensatz and of the real Nullstellensatz for submodules of finitely generated modules and for coherent sheaves, complementing our earlier result that semiprime (resp.\ real) submodules are intersections of prime (resp.\ real prime) submodules. First, we pass from prime to maximal submodules: over a Jacobson ring, a prime submodule $P$ of $M$ is an intersection of maximal submodules whenever $M/P$ is a torsionless $R/(P:M)$-module, in particular whenever $M/P$ is finitely generated (the latter case is due to Maani Shirazi and Sharif). Consequently, if $R$ is a finitely generated algebra over an algebraically closed field and $M$ is finitely generated, the smallest semiprime submodule containing a submodule $N$ of $M$ consists of all $f\in M$ such that, for every point $a$, the value $f(a)$ (i.e.\ the image of $f$ in the fibre $M/\m_aM$) lies in $N(a)$ (i.e.\ the image of $N$ in $M/\m_aM$). We prove the real analogues over real Jacobson rings; for finitely generated algebras over a real closed field they are determined by the real points. The module classifications and directional vanishing formulas yield both Nullstellens\"atze for coherent sheaves on schemes of finite type over a field. We also prove compatibility of semiprimeness, reality and both radicals with arbitrary localization, and the resulting stalk criteria.
\end{abstract}

\maketitle

% First-page section index, in one column.
\begin{center}
\small\textsc{Contents}
\end{center}
\begingroup
\footnotesize
\renewcommand{\arraystretch}{1.15}
\noindent\begin{tabular}{@{}p{\textwidth}@{}}
\hyperref[sec:intro]{\ref{sec:intro}.\ Introduction}\hfill\pageref{sec:intro} \\
\hyperref[sec:definitions]{\ref{sec:definitions}.\ Definitions}\hfill\pageref{sec:definitions} \\
\hyperref[sec:basic]{\ref{sec:basic}.\ Closure properties and affine criteria}\hfill\pageref{sec:basic} \\
\hyperref[sec:class]{\ref{sec:class}.\ Classification of prime and maximal submodules}\hfill\pageref{sec:class} \\
\hyperref[sec:class-sheaves]{\ref{sec:class-sheaves}.\ Classification of prime and maximal subsheaves}\hfill\pageref{sec:class-sheaves} \\
\hyperref[sec:jacobson]{\ref{sec:jacobson}.\ Jacobson properties}\hfill\pageref{sec:jacobson} \\
\hyperref[sec:fibres]{\ref{sec:fibres}.\ Fibre Nullstellens\"atze for modules}\hfill\pageref{sec:fibres} \\
\hyperref[sec:fibres-sheaves]{\ref{sec:fibres-sheaves}.\ Fibre Nullstellens\"atze for sheaves}\hfill\pageref{sec:fibres-sheaves} \\
\hyperref[sec:loc]{\ref{sec:loc}.\ Localization criteria for modules}\hfill\pageref{sec:loc} \\
\hyperref[sec:loc-sheaves]{\ref{sec:loc-sheaves}.\ Stalk criteria for sheaves}\hfill\pageref{sec:loc-sheaves} \\
\hyperref[sec:example]{\ref{sec:example}.\ A projective example}\hfill\pageref{sec:example} \\
\hyperref[sec:questions]{\ref{sec:questions}.\ Questions}\hfill\pageref{sec:questions}
\end{tabular}
\endgroup
\newpage

\section{Introduction}\label{sec:intro}

Throughout, $R$ is a commutative ring with $1$ and modules are unital. We study prime, semiprime and real submodules, which for $M=R$ recover prime, radical and real ideals; the definitions are developed in Section~\ref{sec:definitions}. We write $\rad N$ and $\rrad N$ for the smallest semiprime and real submodules containing $N$. A module is semireal if it has a nonzero real quotient. In \cite{Cdec} we proved that every semiprime submodule of an arbitrary module is an intersection of prime submodules \cite[Theorem~A]{Cdec}, and every real submodule is an intersection of real prime submodules \cite[Theorem~B]{Cdec}; for $M=R^n$, resp.\ finitely generated $M$, the first statement is due to \cite[Theorem~1]{C22a}, resp.\ \cite[Theorem~2.6]{Ary}.

For $M=R$, the decomposition results above are the algebraic halves of Hilbert's Nullstellensatz and of the real Nullstellensatz. Their geometric halves assert that every prime ideal is an intersection of maximal ideals, and every real prime ideal is an intersection of real maximal ideals. These are the \emph{Jacobson} and \emph{real Jacobson} properties, respectively; for modules, the same definitions use submodules in place of ideals. Here \emph{real maximal} means both maximal and real. Finitely generated algebras over fields have both properties \cite[Chapter~4]{Eis} (Proposition~\ref{prop:fgreal} and Remark~\ref{rem:fgreal} for the real case), but both fail in general. Over a non-real base field, the real Jacobson property is vacuous. We prove the corresponding statements for modules and sheaves.

For the sheaf statements, let $X$ be a Noetherian scheme and let $\G\subseteq\F$ be quasi-coherent sheaves. Prime, semiprime and real subsheaves are defined in Section~\ref{sec:definitions}; semiprimeness and reality agree on affine charts with the corresponding module notions (Section~\ref{sec:basic}). We write $\rad\G$ and $\rrad\G$ for the smallest semiprime and real subsheaves containing $\G$.

\begin{mainthm}[Jacobson properties]\label{thmA}
\begin{enumerate}
\item[(a)] Let $P\subsetneq M$ be prime and assume that $M/P$ is torsionless over $R/(P:M)$, for example finitely generated over $R$. If $R$ is Jacobson, then $P$ is an intersection of maximal submodules. Consequently, $R$ is Jacobson if and only if every finitely generated $R$-module is Jacobson.
\item[(b)] Under the same torsionlessness assumption, if $R$ is real Jacobson and $P$ is real prime, then $P$ is an intersection of real maximal submodules. Consequently, $R$ is real Jacobson if and only if every finitely generated $R$-module is real Jacobson.
\item[(c)] Let $X$ be of finite type over a field and $\F$ quasi-coherent. Every prime subsheaf $\mathcal P\subsetneq\F$ with coherent quotient $\F/\mathcal P$ is an intersection of maximal quasi-coherent subsheaves. If $\mathcal P$ is real prime, it is an intersection of real maximal subsheaves.
\end{enumerate}
\end{mainthm}

Part (a) extends the result of Maani Shirazi and Sharif \cite[Proposition~2.9]{MSS} from finitely generated prime quotients to torsionless ones. Both (a) and its real analogue (b) follow from the separation theorem (Theorem~\ref{thm:jac}). The sheaf version in (c) will allow us to test membership in sheaf radicals at closed points (Theorem~\ref{thmB}(c)).

\emph{Fibres and Stalks.} For $\p\in\Spec R$, put $\kappa(\p)=\Frac(R/\p)$ and call $M(\p):=M\otimes_R\kappa(\p)$ the \emph{fibre} of $M$ at $\p$. Write $f(\p)$ for the image of $f\in M$ in $M(\p)$ and $N(\p)$ for the image of $N\otimes_R\kappa(\p)$ in $M(\p)$. For $x\in X$, let $\F_x$ be the \emph{stalk} of $\F$ at $x$, the module of germs of local sections, and let $\m_x$ be the maximal ideal of the local ring $\OO_{X,x}$. Put $\kappa(x):=\OO_{X,x}/\m_x$. The fibre of $\F$ at $x$ is $\F(x):=\F_x/\m_x\F_x$. Write $\G(x)$ for the image of $\G_x$ in $\F(x)$ and $s(x)$ for the image of the germ $s_x$ of a local section $s$. For background, see \cite[II, Sections~1, 2 and~5]{Har}.

Theorem~\ref{thmB} is an instance of the correspondence between submodules and families of fibre subspaces developed in Section~\ref{sec:fibres}.

\begin{mainthm}[Fibre Nullstellens\"atze]\label{thmB}
\begin{enumerate}
\item[(a)] If $M$ is finitely generated and $R$ is Jacobson, then
\[
f\in\rad N\quad\Longleftrightarrow\quad f(\m)\in N(\m)\text{ for every maximal ideal }\m.
\]
\item[(b)] If $M$ is finitely generated and $R$ is real Jacobson, then
\[
f\in\rrad N\quad\Longleftrightarrow\quad f(\m)\in N(\m)\text{ for every real maximal ideal }\m.
\]
\item[(c)] If $X$ is of finite type over a field and $\F$ is coherent, then, on every open $U$,
\[
\begin{aligned}
s\in\Gamma(U,\rad\G)&\Longleftrightarrow s(x)\in\G(x)\text{ for every closed }x\in U,\\
s\in\Gamma(U,\rrad\G)&\Longleftrightarrow s(x)\in\G(x)\text{ for every closed }x\in U
\text{ with }\kappa(x)\text{ real}.
\end{aligned}
\]
\end{enumerate}
In each case the radical is the largest subobject with the same images in the indicated fibres.
\end{mainthm}

Theorem~\ref{thmA} supplies the separation needed to replace arbitrary points by closed points in Theorems~\ref{thm:modfibre}(b) and~\ref{thm:sheaffibre}(b), from which Theorem~\ref{thmB} follows immediately.

For a finitely generated $k$-algebra $R$, write $V_R(k):=\Hom_{k\text{-alg}}(R,k)$ and $\m_a:=\ker a$. If $k$ is algebraically closed in (a), or real closed in (b), it suffices to test $f(a)\in N(a)$ for all $a\in V_R(k)$, where $f(a):=f(\m_a)$ and $N(a):=N(\m_a)$. For $M=R^n$ this is ordinary componentwise evaluation. In (c), the closed points with real residue field are the rational points when the base field is real closed.

In the geometric cases of (a),(b), a proper fibre $N(a)\neq M(a)$ exists if and only if $N\neq M$ in (a), and if and only if $M/N$ is semireal in (b). For a finitely generated $\R$-algebra and $M=R^n$, the real Nullstellensatz was proved in \cite[Theorem~2]{C13}.

\begin{mainthm}[Localization and stalk criteria]\label{thmC}
Let $N\subseteq M$ be modules over a commutative ring $R$. For every multiplicative set $S\subseteq R$,
\[
S^{-1}(\rad N)=\rad(S^{-1}N),\qquad
S^{-1}(\rrad N)=\rrad(S^{-1}N).
\]
Consequently:
\begin{enumerate}
\item[(a)] A submodule $N\subseteq M$ is semiprime, respectively real, if and only if $N_\m\subseteq M_\m$ has that property for every maximal ideal $\m$.
\item[(b)] For $\G\subseteq\F$ quasi-coherent on a Noetherian scheme, semiprimeness and reality are each detected at all stalks (equivalently at all closed stalks). Moreover,
\[
(\rad\G)_x=\rad(\G_x),\qquad
(\rrad\G)_x=\rrad(\G_x),
\]
with the radicals on the right taken inside $\F_x$ over $\OO_{X,x}$.
\end{enumerate}
No finite-generation or coherence assumption is required.
\end{mainthm}

Theorem~\ref{thmC} follows from Proposition~\ref{prop:loc} for modules and Proposition~\ref{prop:stalk} for sheaves.

The next theorem describes every prime submodule or subsheaf by a point and a proper subspace of the fibre at that point. The subspace specifies the allowed values at the chosen point. The resulting submodule or subsheaf is maximal exactly when the point is closed (a maximal ideal in the module case) and the subspace is a hyperplane.

\begin{mainthm}[Classification]\label{thmD}
\begin{enumerate}
\item[(a)] For any $R$-module $M$, the assignment
\[
(\p,W)\longmapsto P_{\p,W}:=\{f\in M:f(\p)\in W\}
\]
is a bijection from prime ideals $\p$ with proper $\kappa(\p)$-subspaces $W\subsetneq M(\p)$ onto prime submodules. Real primes correspond to pairs with real residue field; maximal submodules correspond to maximal $\p$ and hyperplanes $W$.
\item[(b)] For $\F$ quasi-coherent on a Noetherian scheme, $(x,W)\mapsto\mathcal P_{x,W}$, with point kernels defined in Section~\ref{sec:class-sheaves}, is a bijection from points with proper fibre subspaces onto prime quasi-coherent subsheaves. Real primes correspond to points with real residue field. Maximal subsheaves correspond to closed points and hyperplanes.
\end{enumerate}
In both classifications, the real maximal objects correspond to the maximal pairs with real residue field.
\end{mainthm}

Theorem~\ref{thmD} follows from Theorem~\ref{thm:modclass} for modules and Propositions~\ref{prop:primeclass}, \ref{prop:realprimesheaf} and~\ref{prop:maxsheaf} for sheaves. The module classification reformulates classical descriptions by linear equations \cite{THS}; its fibre formulation also leads to the sheaf classification.

Each theme is treated first for modules and then for sheaves: classification and Theorem~\ref{thmD} in Sections~\ref{sec:class}--\ref{sec:class-sheaves}, separation and Theorem~\ref{thmA} in Section~\ref{sec:jacobson}, fibre correspondences and Theorem~\ref{thmB} in Sections~\ref{sec:fibres}--\ref{sec:fibres-sheaves}, and localization and Theorem~\ref{thmC} in Sections~\ref{sec:loc}--\ref{sec:loc-sheaves}. The main proofs use only the elementary properties of localization recalled in Section~\ref{sec:basic}.

\section{Definitions}\label{sec:definitions}

We first express primeness and semiprimeness intrinsically in a Grothendieck category, starting from their module definitions. Reality uses the additional tensor structure. Throughout, a property of a submodule or subsheaf refers to its quotient.

\subsection{Categorical definitions of prime and semiprime objects}

\emph{The module model.} An $R$-module $Q$ is \emph{prime} if $Q\neq0$ and $\Ann_R(q)=\Ann_R(Q)$ for every $0\neq q\in Q$, and \emph{semiprime} if $q\in\Ann_R(q)Q$ implies $q=0$. The latter condition rules out nonzero expressions $q=\sum_j a_jy_j$ whose coefficients all annihilate $q$. For $N\subseteq M$, set
\[
(N:M):=\Ann_R(M/N),\qquad (N:m):=\Ann_R(m+N),
\]
and call $N$ prime or semiprime when $M/N$ is. Equivalently, $N$ is prime if it is proper and $rm\in N$ implies $m\in N$ or $rM\subseteq N$; it is semiprime if $m\in(N:m)M$ implies $m\in N$ \cite[Lemma~2.1]{Cdec}. For $M=R$, these are prime and radical ideals: $\bar r\in\Ann_R(\bar r)(R/I)$ is equivalent to $r\in(I:r)$, hence to $r^2\in I$. This module notion of semiprimeness was introduced for $R^n$ in \cite{C22a} and extended to arbitrary modules in \cite{Ary}.

The following reformulation replaces elements by subobjects.
\begin{proposition}[Submodule criteria]\label{prop:subcrit}
An $R$-module $Q$ is prime if and only if $Q\neq0$ and $\Ann_R(L)=\Ann_R(Q)$ for every nonzero submodule $L\subseteq Q$. It is semiprime if and only if $L\subseteq\Ann_R(L)Q$ implies $L=0$ for every submodule $L\subseteq Q$.
\end{proposition}
\begin{proof}
The submodule conditions imply the elementwise ones by taking $L=Rq$. Conversely, for $0\neq q\in L$, primeness gives $\Ann_R(Q)\subseteq\Ann_R(L)\subseteq\Ann_R(q)=\Ann_R(Q)$, while semiprimeness excludes $q\in\Ann_R(L)Q\subseteq\Ann_R(q)Q$.
\end{proof}

\emph{Closed subcategories and universal quotients.} A Grothendieck category is an abelian category with arbitrary coproducts, exact filtered colimits and a generator; modules and quasi-coherent sheaves are the examples used here. A full subcategory $\mathcal D$ is \emph{closed} if it is closed under subobjects, quotients, arbitrary direct sums and arbitrary direct products \cite[Definitions~2.14 and~2.29]{Kanda22}. Write $\langle Q\rangle_{\mathrm{cl}}$ for the smallest closed subcategory containing $Q$. In $R$-modules, these subcategories are exactly
\[
\mathcal D_I:=\{T:IT=0\},\qquad
\langle L\rangle_{\mathrm{cl}}=\mathcal D_{\Ann_R(L)}
\quad\text{\cite[Proposition~2.32(2),(4)]{Kanda22}}.
\]
Thus $\langle L\rangle_{\mathrm{cl}}$ records the annihilator of $L$. Closed subcategories need not be closed under extensions: modules killed by $2$ form a closed subcategory, but an extension of two such modules can be $\Z/4\Z$.

For every closed $\mathcal D$ there is a largest quotient $q_{\mathcal D}:Q\twoheadrightarrow Q_{\mathcal D}$ belonging to $\mathcal D$: every morphism from $Q$ to an object of $\mathcal D$ factors uniquely through it \cite[Proposition~2.31]{Kanda22}. Concretely, divide $Q$ by the intersection of the kernels of all such morphisms. There is a set of distinct kernels; choosing one morphism for each kernel embeds the resulting quotient into a product of objects of $\mathcal D$, so it belongs to $\mathcal D$. For modules this construction is simply
\[
Q_{\mathcal D_I}=Q/IQ.
\]

\emph{The categorical definitions.} Global sections alone do not detect sheaf radicals (Example~\ref{ex:projective}); we therefore use definitions intrinsic to the category. Let $\mathcal A$ be a Grothendieck category and, for $L\subseteq Q$, write $q_L:Q\twoheadrightarrow Q_{\langle L\rangle_{\mathrm{cl}}}$.
\begin{definition}[Prime and semiprime objects]\label{def:primesheaf}
An object $Q$ is \emph{prime} if $Q\neq0$ and $\langle L\rangle_{\mathrm{cl}}=\langle Q\rangle_{\mathrm{cl}}$ for every nonzero subobject $L\subseteq Q$, following Smith \cite[Definition~4.3]{Smith}, Pappacena \cite[Definition~6.2]{Pap} and Kanda \cite[Definition~4.1]{Kanda22}. Equivalently, $Q\neq0$ and $q_L$ is an isomorphism for every nonzero subobject $L\subseteq Q$.

We call $Q$ \emph{semiprime} if $q_L|_L:L\to Q_{\langle L\rangle_{\mathrm{cl}}}$ is nonzero for every nonzero $L\subseteq Q$. A quasi-coherent subsheaf is prime or semiprime when its quotient is the corresponding object of $\operatorname{QCoh}X$.
\end{definition}
Primeness says that the whole object survives the quotient determined by any nonzero subobject; semiprimeness says that this subobject does not disappear. Every prime object is semiprime. The zero object is semiprime, but not prime. No tensor structure or Noetherian hypothesis is required for these categorical definitions.

\begin{lemma}\label{lem:categoricalsemiprime}
An object $Q$ is semiprime if and only if $q_L|_L$ is a monomorphism for every subobject $L\subseteq Q$.
\end{lemma}
\begin{proof}
Suppose $Q$ is semiprime and $K:=\ker(q_L|_L)\neq0$. Since $\langle K\rangle_{\mathrm{cl}}\subseteq\langle L\rangle_{\mathrm{cl}}$, the map $q_K$ factors through $q_L$. Its restriction to $K$ is therefore zero, contradicting semiprimeness. The converse is immediate.
\end{proof}

\emph{Recovery of the module definitions.} Since $Q_{\langle L\rangle_{\mathrm{cl}}}=Q/\Ann_R(L)Q$, categorical primeness is the prime criterion of Proposition~\ref{prop:subcrit}. Categorical semiprimeness gives its semiprime criterion, or, by Lemma~\ref{lem:categoricalsemiprime}, the equivalent condition
\[
L\cap\Ann_R(L)Q=0\qquad(L\subseteq Q).
\]
For example, $\Z/6\Z$ is semiprime because $6\Z$ is radical, but not prime: the annihilator of $\bar2$ is $3\Z$, whereas the annihilator of the whole module is $6\Z$. For $Q=\Z/4\Z$ and $L=2Q$, we have $\Ann_\Z(L)=2\Z$ and $q_L|_L:L\to Q/2Q$ is zero. Thus $Q$ is not semiprime.

Semiprimeness is not determined by the generated closed subcategory: $\Z/4\Z\oplus\Z$ and $\Z$ both have annihilator zero and hence generate all $\Z$-modules, but only $\Z$ is semiprime. For the former, take $L=(2\Z/4\Z)\oplus0\subseteq2Q=\Ann_\Z(L)Q$. Thus object semiprimeness differs from reducedness of the generated category, as studied in \cite[Definitions~3.7 and~6.6]{Kanda22}.

\emph{The sheaf model.} For a ring $A$, the affine scheme $\Spec A$ has basic opens $D(f)=\{\p:f\notin\p\}$ and structure sheaf given there by $A_f=A[1/f]$. A scheme is obtained by gluing affine schemes, its \emph{affine charts}, along open subsets \cite[II, Section~2]{Har}. It is \emph{Noetherian} if finitely many charts with Noetherian coordinate rings cover it, and \emph{of finite type over $k$} if, as a $k$-scheme, these charts can be chosen with finitely generated $k$-algebras.

The structure sheaves of the charts glue to the \emph{structure sheaf} $\OO_X$; on a chart $U=\Spec A$, we have $\OO_X(D(f))=A_f$. An $\OO_X$-module $\F$ is a sheaf of abelian groups whose sections on each open $V$ form an $\OO_X(V)$-module, with scalar multiplication compatible with restriction. In particular, $\F(D(f))$ is an $A_f$-module on an affine chart.

An $A$-module $M$ defines the sheaf $\widetilde M$ on $\Spec A$ with sections $M_f=M[1/f]$ on $D(f)$; restriction maps are localizations, and compatible sections glue uniquely. An $\OO_X$-module $\F$ is \emph{quasi-coherent} if its restrictions to the affine charts $U_i=\Spec A_i$ are sheaves $\widetilde{M_i}$ associated to $A_i$-modules $M_i$, identified compatibly on overlaps. On our Noetherian schemes, $\F$ is \emph{coherent} if the modules $M_i$ are finitely generated; this can be checked on an affine cover. Throughout the sheaf sections, $X$ is Noetherian.

\emph{The sheaf criteria.} On a Noetherian scheme $X$, let $\mathcal I(\mathcal Q)$ be the largest quasi-coherent ideal sheaf killing $\mathcal Q$, obtained by summing all such ideals. It equals $\Ann\mathcal Q$ for coherent $\mathcal Q$. Closed subcategories of $\operatorname{QCoh}X$ correspond to quasi-coherent ideals $\mathcal I$ through $\mathcal D_{\mathcal I}:=\{\mathcal Q:\mathcal I\mathcal Q=0\}$ \cite[Theorem~9.3(1),(2)]{Kanda22}, \cite[Theorem~11.9 and Lemma~11.8]{Kanda15}; the classification holds even for locally Noetherian schemes. The universal quotient is $\mathcal Q/\mathcal I\mathcal Q$, since every morphism to an object killed by $\mathcal I$ kills $\mathcal I\mathcal Q$. Hence $\mathcal Q_{\langle\mathcal L\rangle_{\mathrm{cl}}}=\mathcal Q/\mathcal I(\mathcal L)\mathcal Q$. For $\mathcal Q=\F/\G$, using Lemma~\ref{lem:categoricalsemiprime}, the definitions become
\begin{equation}\label{eq:sheafcrit}
\begin{aligned}
\G\text{ prime:}\quad &\mathcal Q\neq0,\qquad \mathcal L\neq0\ \Longrightarrow\ \mathcal I(\mathcal L)=\mathcal I(\mathcal Q),\\
\G\text{ semiprime:}\quad &\mathcal L\cap\mathcal I(\mathcal L)\mathcal Q=0
\qquad(\mathcal L\subseteq\mathcal Q\text{ quasi-coherent}).
\end{aligned}
\end{equation}
Equivalently, $\mathcal L\subseteq\mathcal I(\mathcal L)\mathcal Q$ implies $\mathcal L=0$. The ideal-sheaf form is
\begin{equation}\label{eq:sheafidealcrit}
\mathcal J\mathcal L=0,\quad \mathcal L\subseteq\mathcal J\mathcal Q
\quad\Longrightarrow\quad \mathcal L=0
\end{equation}
for every quasi-coherent ideal $\mathcal J$ and every quasi-coherent $\mathcal L\subseteq\mathcal Q$. Indeed, $\mathcal J\subseteq\mathcal I(\mathcal L)$; conversely, take $\mathcal J=\mathcal I(\mathcal L)$.

\begin{remark}[Categorical prime decomposition]\label{rem:catdecomposition}
In a Grothendieck category, every subdirect product of prime objects is semiprime. Conversely, if every nonzero semiprime object contains a prime subobject, every semiprime object is a subdirect product of prime objects; see also \cite{Cnc}. For a locally Noetherian scheme $X$, every nonzero object of $\operatorname{QCoh}X$ contains a prime subobject \cite[Proposition~9.10]{Kanda22}. Hence every semiprime quasi-coherent sheaf is a subdirect product of prime quasi-coherent sheaves. The results below give the explicit fibre description, localization criteria and real decomposition beyond this abstract separation theorem.
\end{remark}

\begin{proof}
Let $Q\subseteq\prod_iQ_i$ with all $Q_i$ prime, hence semiprime, let $L\subseteq Q$, put $\mathcal C=\langle L\rangle_{\mathrm{cl}}$, and let $L_i$ be the image of $L$ in $Q_i$. As $\langle L_i\rangle_{\mathrm{cl}}\subseteq\mathcal C$, the map $L_i\to(Q_i)_{\langle L_i\rangle_{\mathrm{cl}}}$ factors through $L_i\to(Q_i)_{\mathcal C}$, which is therefore a monomorphism by Lemma~\ref{lem:categoricalsemiprime}. Since $\prod_i(Q_i)_{\mathcal C}\in\mathcal C$, the map $Q\to\prod_i(Q_i)_{\mathcal C}$ factors through $q_L$, and its restriction to $L$ is the monomorphism $L\to\prod_iL_i\to\prod_i(Q_i)_{\mathcal C}$. Hence $q_L|_L$ is a monomorphism, and $Q$ is semiprime.

Conversely, subobjects of a semiprime $Q$ are semiprime: for $L\subseteq Y\subseteq Q$, the map $L\to Q_{\langle L\rangle_{\mathrm{cl}}}$ factors through $Y_{\langle L\rangle_{\mathrm{cl}}}$. Let $L\subseteq Q$ be a prime subobject, $\mathcal C=\langle L\rangle_{\mathrm{cl}}$ and $K_0=\ker q_L$, so $L\cap K_0=0$ by Lemma~\ref{lem:categoricalsemiprime}. By Zorn's lemma and exactness of filtered colimits, choose $P\supseteq K_0$ maximal with $P\cap L=0$. Then $Q/P$, a quotient of $Q_{\mathcal C}$, lies in $\mathcal C$. If $0\neq K/P\subseteq Q/P$, maximality gives $K\cap L\neq0$, and $K\cap L$ embeds into $K/P$. Nonzero subobjects of prime objects are prime \cite[Proposition~4.2(1)]{Kanda22}, so $\mathcal C=\langle K\cap L\rangle_{\mathrm{cl}}\subseteq\langle K/P\rangle_{\mathrm{cl}}\subseteq\mathcal C$; thus $Q/P$ is prime. Let $Y$ be the intersection of all such $P$. If $Y\neq0$, it is semiprime and contains a prime subobject $L$; the corresponding $P$ satisfies $L\subseteq Y\subseteq P$ and $P\cap L=0$, a contradiction. Hence $Q\to\prod Q/P$ is a monomorphism.
\end{proof}

\subsection{Real modules and sheaves}

Reality uses the tensor structure of modules and sheaves, in addition to their abelian-category structure. An $R$-module $Q$ is \emph{real} if $\sum_i q_i\otimes q_i=0$ in $Q\otimes_R Q$ implies $q_i=0$ for every $i$; a submodule $N\subseteq M$ is real when $M/N$ is real. Real modules are semiprime \cite[Lemma~4.1]{Cdec}.

For $M=R$, an ideal $I$ is \emph{real} if $\sum_ia_i^2\in I$ implies $a_i\in I$ for all $i$. The $R$-module $R/I$ is real if and only if $I$ is a real ideal, since $R/I\otimes_RR/I\cong R/I$ via $\bar a\otimes\bar b\mapsto\overline{ab}$, which maps $\sum_i\bar a_i\otimes\bar a_i$ to $\overline{\sum_ia_i^2}$. Thus real (prime) submodules of $R$ are exactly real (prime) ideals.

To express the same condition without elements, in either category let $\pi_i:Q^{\oplus n}\to Q$ be the projections and put
\[
b_n:=\sum_{i=1}^n\pi_i\otimes\pi_i:
Q^{\oplus n}\otimes Q^{\oplus n}\longrightarrow Q\otimes Q.
\]
A subobject $L\subseteq Q^{\oplus n}$ is \emph{totally isotropic} if the composite $L\otimes L\to Q^{\oplus n}\otimes Q^{\oplus n}\xrightarrow{b_n}Q\otimes Q$ is zero. Tensor products are over $R$ for modules and over $\OO_X$ for sheaves; the tensor product of an inclusion need not be injective. For modules, reality is equivalent to the absence of nonzero totally isotropic submodules for every $n$: a tuple with $\sum_i q_i\otimes q_i=0$ generates a totally isotropic cyclic submodule. Conversely, if $L$ is totally isotropic, every $q=(q_i)\in L$ satisfies $\sum_i q_i\otimes q_i=b_n(q\otimes q)=0$, so reality gives $q=0$.

\begin{definition}\label{def:realsheaf}
A quasi-coherent subsheaf $\G\subseteq\F$ is \emph{real} if the only totally isotropic quasi-coherent subsheaf of $(\F/\G)^{\oplus n}$ is zero for every $n\ge1$. A submodule or quasi-coherent subsheaf is \emph{real prime} or \emph{real maximal} if it is real and prime, or real and maximal, respectively. A module or quasi-coherent sheaf is \emph{semireal} if it has a nonzero real quotient.
\end{definition}

A submodule or quasi-coherent subsheaf is \emph{maximal} if it is proper and has no proper intermediate submodule, resp.\ quasi-coherent subsheaf. We write $\rad N$, $\rrad N$ and $\rad\G$, $\rrad\G$ for the intersections of the semiprime and real subobjects containing the given subobject. For modules these are the smallest such submodules by Proposition~\ref{prop:closprop}; for sheaves, Theorem~\ref{thm:sheaffibre} proves that these intersections are quasi-coherent and gives their affine descriptions. Quasi-coherence here requires proof: arbitrary intersections taken as sheaves need not preserve it.

\section{Closure properties and affine criteria}\label{sec:basic}

We collect the closure properties and affine comparisons needed in the proofs. The module statements apply over arbitrary commutative rings; throughout the sheaf part, $X$ is Noetherian.

\subsection{Modules}

\begin{proposition}[Closure properties]\label{prop:closprop}
\begin{enumerate}
\item[(a)] Submodules and direct products (hence direct sums) of semiprime or real modules have the same property.
\item[(b)] Inverse images and intersections of semiprime or real submodules have the same property.
\end{enumerate}
\end{proposition}
\begin{proof}
For submodules, inclusions carry each self-annihilating expression or vanishing tensor sum to the ambient module; for products, the coordinate projections detect each such expression or sum. Thus (a) holds. For (b), a homomorphism $f:M\to M'$ induces $M/f^{-1}(N')\hookrightarrow M'/N'$, and $M/\bigcap_iN_i$ embeds into $\prod_iM/N_i$; apply (a).
\end{proof}

Consequently, intersections define the smallest semiprime and real submodules containing $N$, namely $\rad N$ and $\rrad N$. If $f:M\to M'$ satisfies $f(N)\subseteq N'$, applying the inverse-image property gives $f(\rad N)\subseteq\rad N'$ and $f(\rrad N)\subseteq\rrad N'$. By Proposition~\ref{prop:closprop}(a) and the coordinate inclusions,
\begin{equation}\label{eq:radcoproduct}
\rad\Bigl(\bigoplus_iN_i\Bigr)=\bigoplus_i\rad N_i,
\qquad
\rrad\Bigl(\bigoplus_iN_i\Bigr)=\bigoplus_i\rrad N_i.
\end{equation}

\emph{Real annihilators and semireal rings.} If $N$ is a real submodule, then every $(N:m)$ is a real ideal \cite[Proposition~4.3(a)]{Cdec}, so if $P$ is a real prime submodule, then $(P:M)$ is a real prime ideal. A ring is \emph{semireal} if $-1$ is not a sum of squares, equivalently if it has a proper real ideal, equivalently a real prime ideal \cite[Theorem~2.3]{LamReal}; by the definition in Section~\ref{sec:definitions}, this agrees with semireality of $R$ as an $R$-module. For an ideal $I$, the real prime (resp.\ real maximal) ideals of $R/I$ are the $\p/I$ with $\p\supseteq I$ a real prime (resp.\ real maximal) ideal of $R$, since $(R/I)/(\p/I)\cong R/\p$. In particular, $R/I$ is semireal if and only if some real prime ideal of $R$ contains $I$.

\subsection{Quasi-coherent sheaves}

Let $\G\subseteq\F$ be quasi-coherent sheaves. On an affine chart $U=\Spec A$, write $\F|_U=\widetilde M$ and $\G|_U=\widetilde N$, with $N\subseteq M$. Recall $M_f=M[1/f]$, and put $M_\p=(A\setminus\p)^{-1}M$. Localization preserves exact sequences and commutes with tensor products \cite[Chapter~3]{AM}.

\emph{The affine dictionary.} For $x\in U$ corresponding to $\p$, and $s\in\Gamma(U,\F)$ corresponding to $m\in M$:
\begin{center}\small\renewcommand{\arraystretch}{1.35}
\begin{tabular}{@{}>{\raggedright}p{0.29\textwidth}>{\raggedright\arraybackslash}p{0.64\textwidth}@{}}
\textbf{sheaf construction} & \textbf{module description on $U=\Spec A$} \\ \hline
sections on a basic open & $\Gamma(D(f),\F)=M_f$ \\
local ring and stalk at $x$ & $\OO_{X,x}=A_\p$, $\F_x=M_\p$; the germ $s_x$ is $m/1$ \\
residue field & $\kappa(x)=A_\p/\p A_\p=\Frac(A/\p)$ \\
fibre and value at $x$ & $\F(x)=M\otimes_A\kappa(x)$; $s(x)=m\otimes1$ \\
quotient & $(\F/\G)|_U=\widetilde{M/N}$ \\
tensor product & $(\F\otimes\mathcal E)|_U=\widetilde{M\otimes_A L}$ if $\mathcal E|_U=\widetilde L$
\end{tabular}
\end{center}
Passing to a stalk localizes; passing to a fibre then reduces modulo the maximal ideal. If $x\in U=\Spec A$ is closed in $U$, its corresponding prime ideal $\m\subseteq A$ is maximal. Then $\kappa(x)=A/\m$ and $\F(x)=M/\m M$. A point closed in $X$ is closed in every affine neighbourhood, but a point closed in one affine neighbourhood need not be closed in $X$: the generic point of the spectrum of a discrete valuation ring is closed in the open subset it forms. Write $\G(x)$ for the image of $\G_x$ in $\F(x)$: the map from the fibre of $\G$ to that of $\F$ need not be injective. Kernels, images and quotients of quasi-coherent sheaves are computed on affine opens by the corresponding module operations, and everywhere by stalks. Inclusion and equality can be checked on stalks; quasi-coherent subsheaves and quotients of coherent sheaves are coherent. See \cite[II, Sections~1, 2 and~5]{Har}.

\emph{Extension.} On a Noetherian scheme, local data can be extended without leaving a given ambient sheaf: if $\F$ is quasi-coherent, $U\subseteq X$ is open and $\mathcal E_U\subseteq\F|_U$ is coherent, then there is a coherent $\mathcal E\subseteq\F$ with $\mathcal E|_U=\mathcal E_U$ \cite[Tag~01PF]{Stacks}. On an affine $U=\Spec A$, this applies to $\mathcal E_U=\widetilde C$ for every finitely generated submodule $C\subseteq\Gamma(U,\F)$, and gives $\Gamma(U,\mathcal E)=C$. For coherent $\mathcal E$, the annihilator $\Ann\mathcal E$ is a coherent ideal and $\Gamma(U,\Ann\mathcal E)=\Ann_A\Gamma(U,\mathcal E)$ on affine $U$.

The intrinsic sheaf criteria agree with the module criteria on affine charts.

\begin{proposition}[Affine criteria]\label{prop:sheafaffine}
Let $X$ be Noetherian and $\G\subseteq\F$ quasi-coherent.
\begin{enumerate}
\item[(a)] $\G$ is semiprime if and only if $\Gamma(U,\G)$ is a semiprime submodule of $\Gamma(U,\F)$ on every affine open $U$, if and only if this holds on an affine open cover.
\item[(b)] The same equivalences hold for reality. Every real subsheaf is semiprime.
\end{enumerate}
\end{proposition}
\begin{proof}
Put $\mathcal Q=\F/\G$; on an affine $U=\Spec A$, $\Gamma(U,\mathcal Q)=\Gamma(U,\F)/\Gamma(U,\G)$. The implications from an affine cover to the global conditions follow by checking sections on that cover. For the converse implications, we turn a counterexample on one affine open into a subsheaf forbidden by the global definition, using the extension fact above.

(a) Suppose $\G$ is semiprime, let $U=\Spec A$ be affine, $M=\Gamma(U,\mathcal Q)$ and $q\in\Ann_A(q)M$. Extend $Aq\subseteq M$ to a coherent subsheaf $\mathcal E\subseteq\mathcal Q$ with $\Gamma(U,\mathcal E)=Aq$. As $\mathcal E$ is coherent, $\mathcal I(\mathcal E)=\Ann\mathcal E$, and $\Gamma(U,\Ann\mathcal E)=\Ann_A(Aq)=\Ann_A(q)$. Hence
\[
q\in Aq\cap\Ann_A(q)M=\Gamma\bigl(U,\ \mathcal E\cap\mathcal I(\mathcal E)\mathcal Q\bigr).
\]
This is the module obstruction $L\cap\Ann_A(L)M$ for $L=Aq$, and semiprimeness in the form \eqref{eq:sheafcrit} gives $\mathcal E\cap\mathcal I(\mathcal E)\mathcal Q=0$. Thus $q=0$, and $\Gamma(U,\G)$ is a semiprime submodule of $\Gamma(U,\F)$.

Conversely, suppose that $\Gamma(U_i,\mathcal Q)$ is semiprime for every member $U_i=\Spec A_i$ of an affine open cover, and let $\mathcal L\subseteq\mathcal I(\mathcal L)\mathcal Q$. On $U_i$, the sections of $\mathcal I(\mathcal L)\mathcal Q$ are $\Gamma(U_i,\mathcal I(\mathcal L))\Gamma(U_i,\mathcal Q)$, and $\Gamma(U_i,\mathcal I(\mathcal L))$ kills $\Gamma(U_i,\mathcal L)$. Hence every $q\in\Gamma(U_i,\mathcal L)$ satisfies $q\in\Ann_{A_i}(q)\Gamma(U_i,\mathcal Q)$, so $q=0$. Thus $\mathcal L=0$.

(b) Suppose $\G$ is real. On an affine open $U=\Spec A$, put $Q=\Gamma(U,\mathcal Q)$ and suppose $\sum_i q_i\otimes q_i=0$ in $Q\otimes_A Q$. The cyclic submodule $Aq\subseteq Q^n$ generated by $q=(q_1,\dots,q_n)$ is totally isotropic, since $b_n(aq\otimes cq)=ac\sum_iq_i\otimes q_i=0$; let $\mathcal L_U\subseteq\mathcal Q^{\oplus n}|_U$ be the corresponding coherent subsheaf. Extend it to a coherent $\mathcal E\subseteq\mathcal Q^{\oplus n}$. The extension need not remain totally isotropic outside $U$. Let
\[
\mathcal B:=\operatorname{im}\bigl(\mathcal E\otimes\mathcal E\xrightarrow{b_n}\mathcal Q\otimes\mathcal Q\bigr)
\]
measure this failure; it is coherent, being a quotient of the coherent sheaf $\mathcal E\otimes\mathcal E$, although $\mathcal Q\otimes\mathcal Q$ need not be coherent. Set $\mathcal J:=\Ann\mathcal B$. Since $\mathcal B|_U=0$, we have $\mathcal J|_U=\OO_U$. Consequently $\mathcal L:=\mathcal J\mathcal E$ agrees with $\mathcal L_U$ on $U$, while its pairing vanishes globally: locally, $b_n(je\otimes j'e')=jj'\,b_n(e\otimes e')=0$, because $\mathcal J\mathcal B=0$. Reality gives $\mathcal L=0$, so $q_i=0$ for all $i$. Thus $Q$ is real.

Conversely, suppose the quotient modules on the members of an affine open cover are real and $\mathcal L\subseteq\mathcal Q^{\oplus n}$ is totally isotropic. On a member $U$, a section $q=(q_1,\dots,q_n)$ of $\mathcal L$ satisfies $b_n(q\otimes q)=\sum_i q_i\otimes q_i=0$ in $\Gamma(U,\mathcal Q\otimes\mathcal Q)=\Gamma(U,\mathcal Q)\otimes\Gamma(U,\mathcal Q)$, hence is zero. Thus $\mathcal L=0$.

Finally, real modules are semiprime: if $q=\sum_ja_jy_j$ with $a_jq=0$, then $q\otimes q=\sum_jy_j\otimes a_jq=0$, so $q=0$. Hence part (a) shows that every real subsheaf is semiprime.
\end{proof}

\section{Classification of prime and maximal submodules}\label{sec:class}

We classify prime, real prime and maximal submodules: a prime submodule is the inverse image of a proper subspace of a single fibre. Their sheaf counterparts follow in Section~\ref{sec:class-sheaves}.

\begin{definition}[Fibres and point kernels of a module]\label{def:modfibres}
For $\p\in\Spec R$, put $\kappa(\p)=\Frac(R/\p)$ and define the \emph{fibre} $M(\p):=M\otimes_R\kappa(\p)$. The \emph{fibre map} is the natural $R$-linear map $M\to M(\p)$, $m\mapsto m(\p):=m\otimes1$. For a submodule $N\subseteq M$, define $N(\p)$ as the image of $N\otimes_R\kappa(\p)$ in $M(\p)$.
For a subspace $W\subseteq M(\p)$, define the \emph{point kernel}
\begin{equation}\label{eq:modulepointkernel}
P_{\p,W}=\{m\in M:m(\p)\in W\}.
\end{equation}
The \emph{saturation} of $N+\p M$ is
\begin{equation}\label{eq:modulesaturation}
K(N,\p):=\{m\in M:sm\in N+\p M\text{ for some }s\notin\p\}.
\end{equation}
\end{definition}

\begin{lemma}[Fibre identities]\label{lem:fibreid}
Let $\p\in\Spec R$ and $N\subseteq M$.
\begin{enumerate}
\item[(a)] Reduction modulo $\p$ and localization give
\[
M(\p)\cong(M/\p M)\otimes_{R/\p}\kappa(\p)\cong M_\p/\p M_\p.
\]
\item[(b)] The image of $N$ in the fibre is
\[
N(\p)=\operatorname{span}_{\kappa(\p)}\{n\otimes1:n\in N\}
=\frac{N_\p+\p M_\p}{\p M_\p}.
\]
\item[(c)] Taking inverse images gives the saturation, and for $N=0$ the kernel of the fibre map:
\[
P_{\p,N(\p)}=K(N,\p),\qquad
\ker\bigl(M\to M(\p)\bigr)=K(0,\p)=\{m\in M:sm\in\p M\text{ for some }s\notin\p\}.
\]
\item[(d)] If $\p=\m$ is maximal, then $\kappa(\m)=R/\m$, $M(\m)=M/\m M$, the fibre map is the quotient map, $N(\m)=(N+\m M)/\m M$, and
\[
K(N,\m)=P_{\m,N(\m)}=N+\m M.
\]
\end{enumerate}
\end{lemma}

\begin{proof}
(a) The ideal $\p$ acts trivially on $\kappa(\p)$, and localization commutes with quotients. (b) The tensor image is spanned by the $n\otimes1$; in $M_\p/\p M_\p$ it is the image of $N_\p$. (c) By (b), $m(\p)\in N(\p)$ means $m/1\in N_\p+\p M_\p$, equivalently $sm\in N+\p M$ for some $s\notin\p$. (d) Here $R/\p$ is already a field; apply (c).
\end{proof}

\begin{lemma}[Prime quotients]\label{lem:primequot}
Let $P$ be a proper submodule of an $R$-module $M$ and $\p:=(P:M)$. Then $P$ is a prime submodule if and only if $\p$ is a prime ideal and $M/P$ is a torsion-free $R/\p$-module \cite[Lemma~1]{THS}. In this case $P=K(P,\p)$, and
\[
P\text{ is a real submodule}\iff\p\text{ is a real ideal}\iff\kappa(\p)\text{ is a real field}.
\]
\end{lemma}

\begin{proof}
Torsion-freeness of $M/P$ over $R/\p$ says that $rm\in P$ and $m\notin P$ imply $r\in\p$, i.e.\ $rM\subseteq P$; given $P\neq M$, this is primeness of $P$. It remains to see that $\p$ is a prime ideal if $P$ is a prime submodule: $\p\neq R$, and if $ab\in\p$ and $b\notin\p$, choose $m$ with $bm\notin P$; then $a(bm)\in P$ gives $a\in\p$. If $sm\in P+\p M=P$ with $s\notin\p$, then $m\in P$ by torsion-freeness; so $P=K(P,\p)$.

Put $Q=M/P$ and $A=R/\p$, so that $\kappa(\p)=\Frac(A)$. If $Q$ is real and $\sum_i a_i^2=0$ in $A$, choose $0\neq q\in Q$. Then $\sum_i(a_iq)\otimes(a_iq)=0$, so every $a_iq=0$; torsion-freeness gives $a_i=0$. Thus $A$ is real. Clearing denominators shows that $A$ is real if and only if $\kappa(\p)$ is real.

Conversely, suppose $\kappa(\p)$ is real and $\sum_iq_i\otimes q_i=0$ in $Q\otimes_R Q$. In $V=Q\otimes_A \kappa(\p)$, applying $\lambda\otimes\lambda$ for any $\kappa(\p)$-linear functional $\lambda:V\to \kappa(\p)$ gives $\sum_i\lambda(q_i\otimes1)^2=0$. Hence every functional vanishes on each $q_i\otimes1$, so these vectors are zero. Since $Q\to V$ is injective by torsion-freeness, all $q_i=0$. Thus $Q$ is real.
\end{proof}

\begin{lemma}\label{lem:simple}
For a submodule $P\subsetneq M$ the following are equivalent: \textup{(i)} $P$ is a maximal submodule; \textup{(ii)} $M/P$ is simple; \textup{(iii)} $M/P\cong R/\m$ for some maximal ideal $\m$. In this case $\m=(P:M)=(P:x)$ for every $x\in M\setminus P$, and $P$ is a prime submodule.
\end{lemma}

\begin{proof}
(i)$\Leftrightarrow$(ii) is the correspondence between submodules of $M$ containing $P$ and submodules of $M/P$. A simple module is generated by each of its nonzero elements $\bar x$, so it is isomorphic to $R/\Ann_R(\bar x)$; and $R/I$ is simple if and only if $I$ is a maximal ideal. This gives (ii)$\Leftrightarrow$(iii). In (iii), $M/P$ is a one-dimensional $R/\m$-vector space, so every nonzero element has annihilator $\m=\Ann_R(M/P)$; that is, $(P:x)=(P:M)=\m$ for $x\notin P$, and $P$ is a prime submodule.
\end{proof}

\begin{theorem}[Classification of submodules]\label{thm:modclass}
Let $R$ be a commutative ring and $M$ any $R$-module.
\begin{enumerate}
\item[(a)] $(\p,W)\mapsto P_{\p,W}$ is a bijection from pairs with $\p\in\Spec R$ and $W\subsetneq M(\p)$ onto the prime submodules of $M$. Its inverse sends $P$ to $\p=(P:M)$ and $W=P(\p)$.
\item[(b)] This bijection restricts to real prime submodules by requiring $\kappa(\p)$ real.
\item[(c)] It restricts to maximal submodules by requiring $\p$ maximal and $W$ a hyperplane, and to real maximal submodules by additionally requiring $\kappa(\p)$ real.
\end{enumerate}
\end{theorem}

\begin{proof}
(a) Put $P=P_{\p,W}$. The map $M\to M(\p)/W$, $m\mapsto m(\p)+W$, has kernel $P$ and nonzero image, since the values $m(\p)$ span $M(\p)$ and $W$ is proper. Thus $M/P$ embeds as a nonzero submodule of $M(\p)/W$. It is therefore torsion-free over $R/\p$; hence $\Ann_R(M/P)=\p$, and Lemma~\ref{lem:primequot} shows that $P$ is prime. Clearly $P(\p)\subseteq W$. For the reverse inclusion, every $w\in W$ has the form $m(\p)/\bar s$ with $s\notin\p$. Then $m(\p)=\bar s w\in W$, so $m\in P$. Thus $P(\p)=W$.

Conversely, let $P$ be prime and $\p=(P:M)$. Lemma~\ref{lem:primequot} and Lemma~\ref{lem:fibreid}(c) give
\[
P=K(P,\p)=P_{\p,P(\p)}.
\]
The subspace $P(\p)$ is proper because $M/P$ is nonzero and torsion-free over $R/\p$, hence embeds into its nonzero fibre $M(\p)/P(\p)$. The two constructions are therefore inverse.

(b) Lemma~\ref{lem:primequot} says precisely that a prime submodule with annihilator $\p$ is real if and only if $\kappa(\p)$ is real.

(c) By Lemma~\ref{lem:simple}, a maximal $P$ is prime and $M/P\cong R/\m$ for the maximal ideal $\m=(P:M)$. Thus its quotient is one-dimensional over $R/\m$. Conversely, for maximal $\m$, the fibre map $M\to M(\m)=M/\m M$ is surjective, so
\[
M/P_{\m,W}\cong M(\m)/W.
\]
This quotient is simple exactly when $W$ is a hyperplane; by (b), a maximal submodule is real exactly when its maximal residue field is real. No finite-generation assumption is required. This proves Theorem~\ref{thmD}(a).
\end{proof}

The prime classification is a fibre formulation of classical descriptions by torsion-free quotients and linear equations; see \cite[Lemma~6, Theorems~8 and~10]{THS}. We include the short proof to make the correspondence $P\leftrightarrow(\p,P(\p))$ explicit.

\begin{remark}[Prime quotients]\label{rem:primequotients}
For a prime $\p$ and $W\subsetneq M(\p)$, put $P=P_{\p,W}$. Since $\p M\subseteq P$ and $P(\p)=W$, localization and passage to the fibre give
\begin{equation}\label{eq:primefibrequotient}
(M/P)_\p\cong(M/P)\otimes_R\kappa(\p)\cong M(\p)/W.
\end{equation}
The natural map $M/P\hookrightarrow M(\p)/W$ need not be surjective: for $R=M=\mathbb Z$, $\p=(0)$ and $W=0$ it is $\mathbb Z\hookrightarrow\mathbb Q$. For maximal $\p=\m$, it is an isomorphism.

\end{remark}

\section{Classification of prime and maximal subsheaves}\label{sec:class-sheaves}

The classifications below prove Theorem~\ref{thmD}(b). Prime quotients may be attached to arbitrary points; maximal quotients are attached to closed points. No coherence assumption on $\F$ is needed.

\emph{Point sheaves.} For $x\in X$ let $j_x:\Spec\kappa(x)\to X$ be the canonical morphism, induced on each affine neighbourhood $\Spec A$ of $x\leftrightarrow\p$ by $A\to A_\p\to\kappa(x)$. For a $\kappa(x)$-vector space $V$, the point sheaf $(j_x)_*V$ has sections $V$ on opens containing $x$ and zero elsewhere; for closed $x$ it is the usual skyscraper sheaf. For a subspace $W\subseteq\F(x)$, define the \emph{point kernel}
\[
\mathcal P_{x,W}:=\ker\bigl(\F\longrightarrow(j_x)_*(\F(x)/W)\bigr).
\]

Since $j_x^*\F=\F_x\otimes_{\OO_{X,x}}\kappa(x)=\F(x)$, adjunction gives
\[
\Hom_{\OO_X}\bigl(\F,(j_x)_*V\bigr)
\cong\Hom_{\kappa(x)}\bigl(\F(x),V\bigr).
\]
Thus a morphism $\F\to(j_x)_*V$ corresponds to a linear map $\lambda:\F(x)\to V$ and has kernel $\mathcal P_{x,\ker\lambda}$. Its image is isomorphic to $\F/\mathcal P_{x,\ker\lambda}$ and has stalk $\operatorname{im}\lambda$ at $x$.

A prime quotient has a single attached point. This is the global step beyond the module classification. On an affine open $U=\Spec A$ containing $\eta$, write $\p_\eta\subseteq A$ for the prime ideal corresponding to $\eta$.

\begin{lemma}[Affine sections of prime sheaves]\label{lem:primesections}
Let $\mathcal Q$ be a prime quasi-coherent sheaf on $X$. There is a point $\eta\in X$ such that, for every affine open $U=\Spec A$, $\Gamma(U,\mathcal Q)\neq0$ if and only if $\eta\in U$; in this case $\Gamma(U,\mathcal Q)$ is a prime $A$-module with annihilator $\p_\eta$. Thus prime subsheaves have prime or whole-module sections on affine opens.
\end{lemma}

\begin{proof}
Let $U=\Spec A$, $M=\Gamma(U,\mathcal Q)$ and $0\neq q\in M$. Since $A$ is Noetherian, $\widetilde{Aq}$ is coherent. By the extension fact in Section~\ref{sec:basic}, choose a coherent $\mathcal E\subseteq\mathcal Q$ with $\mathcal E|_U=\widetilde{Aq}$. Since $\mathcal E\neq0$, the prime criterion in \eqref{eq:sheafcrit} gives $\mathcal I(\mathcal Q)=\mathcal I(\mathcal E)=\Ann\mathcal E$, and $\Gamma(U,\Ann\mathcal E)=\Ann_A(q)$. Hence
\[
\Ann_A(q)=\Gamma(U,\mathcal I(\mathcal Q))
\subseteq\Ann_A(M)\subseteq\Ann_A(q).
\]
Thus $M$ is prime whenever it is nonzero, with prime annihilator $\p_U=\Gamma(U,\mathcal I(\mathcal Q))$. Let $\eta_U$ correspond to $\p_U$ and let $Z_U$ be its reduced closure in $X$.

Let $\mathcal T\subseteq\mathcal Q$ be the subsheaf of local sections annihilated by $\mathcal I_{Z_U}$. Since $\mathcal I_{Z_U}$ is coherent, $\mathcal T$ is quasi-coherent. Moreover, $\mathcal T|_U=\mathcal Q|_U\neq0$, because $\mathcal I_{Z_U}|_U=\widetilde{\p_U}$ annihilates $\mathcal Q|_U$. Primeness therefore gives
\[
\mathcal I_{Z_U}\subseteq\mathcal I(\mathcal T)=\mathcal I(\mathcal Q).
\]
For another affine $U'$ with nonzero sections, taking sections gives $\Gamma(U',\mathcal I_{Z_U})\subseteq\p_{U'}$, hence $\eta_{U'}\in Z_U$; symmetrically $\eta_U\in Z_{U'}$. The closures, and hence their generic points, coincide. Call the common point $\eta$. By Lemma~\ref{lem:primequot}, $M$ is torsion-free over $A/\p_U$, so it embeds into $M_{\p_U}=\mathcal Q_\eta\neq0$. Every affine open containing $\eta$ therefore has nonzero sections.
\end{proof}

\begin{proposition}[Classification of prime subsheaves]\label{prop:primeclass}
Let $x\in X$ and $W\subseteq\F(x)$ be a $\kappa(x)$-subspace.
\begin{enumerate}
\item[(a)] $\mathcal P_{x,W}$ is quasi-coherent. Its sections over $U$ are the sections $s$ with $s(x)\in W$ if $x\in U$, and all sections of $\F$ otherwise. Moreover, $\G\subseteq\mathcal P_{x,W}$ if and only if $\G(x)\subseteq W$. On an affine $U=\Spec A$ containing $x\leftrightarrow\p$, with $M=\Gamma(U,\F)$ and $N=\Gamma(U,\G)$,
\[
\Gamma(U,\mathcal P_{x,\G(x)})=P_{\p,N(\p)}=K(N,\p).
\]

\item[(b)] The assignment $(x,W)\mapsto\mathcal P_{x,W}$ is a bijection from pairs with $W\subsetneq\F(x)$ onto the prime quasi-coherent subsheaves of $\F$. Its inverse sends $\mathcal P$ to the point $x$ attached to $\F/\mathcal P$ in Lemma~\ref{lem:primesections} and $W=\mathcal P(x)$.
\end{enumerate}
\end{proposition}

\begin{proof}
(a) On any affine open containing $x$, the target is the sheaf associated to the $A$-module $V=\F(x)/W$, with scalars acting through $A\to\kappa(x)$; on an affine open not containing $x$, it is zero. Thus the target and the kernel are quasi-coherent. The section description follows from the map to the fibre. For $W=\G(x)$, its module of sections is $P_{\p,N(\p)}=K(N,\p)$ by Lemma~\ref{lem:fibreid}(c). The containment assertion follows from the fibre map.

(b) Let $W\subsetneq\F(x)$, and let $\mathcal Q\subseteq(j_x)_*V$, $V=\F(x)/W$, be the image of $\F\to(j_x)_*V$; it is nonzero, as its stalk at $x$ is $V$. We show that $\mathcal Q\cong\F/\mathcal P_{x,W}$ is prime with attached point $x$. Let $Z$ be the reduced closure of $x$; then $\mathcal I_Z(j_x)_*V=0$. For $0\neq\mathcal L\subseteq\mathcal Q$, choose an affine $U=\Spec A$ with nonzero sections. Necessarily $x\in U$, and every nonzero section is a vector in $V$ with annihilator $\p_x$. It remains nonzero at $x$, since $A\setminus\p_x$ acts invertibly on $V$; hence $\mathcal L$ has nonzero sections on every affine neighbourhood of $x$. Any quasi-coherent ideal sheaf $\mathcal J$ killing $\mathcal L$ therefore satisfies $x\in V(\mathcal J)$, hence $\mathcal J\subseteq\mathcal I_Z$: on affines containing $x$ this follows from $\Gamma(U,\mathcal J)\subseteq\p_x$, and on opens not containing $x$ we have $\mathcal I_Z=\OO_X$, since every open meeting $Z$ contains its generic point $x$. Thus $\mathcal I(\mathcal L)=\mathcal I_Z$, and \eqref{eq:sheafcrit} proves primeness. The point attached to $\mathcal Q$ is $x$, since nonzero sections embed into their localization at $\p_x$.

Conversely, let $\mathcal Q=\F/\mathcal P$ be prime, with attached point $x$ from Lemma~\ref{lem:primesections}. On each affine $U=\Spec A$ containing $x$, $\Gamma(U,\mathcal P)$ is a prime submodule of $\Gamma(U,\F)$ with annihilator $\p_x$. Theorem~\ref{thm:modclass}(a) identifies it with the inverse image of $\mathcal P(x)$ in the fibre at $x$. On affines not containing $x$, $\mathcal Q$ has zero sections, so $\mathcal P$ equals $\F$. These affine descriptions give $\mathcal P=\mathcal P_{x,\mathcal P(x)}$.

Finally, $\mathcal P_{x,W}(x)=W$, because $\F_x\to\F(x)/W$ is surjective. Together with recovery of the attached point, this proves uniqueness of the pair.
\end{proof}

\begin{proposition}[Classification of real prime subsheaves]\label{prop:realprimesheaf}
The correspondence $(x,W)\mapsto\mathcal P_{x,W}$ restricts to a bijection between pairs with $\kappa(x)$ real and $W\subsetneq\F(x)$ a $\kappa(x)$-subspace, and the real prime subsheaves of $\F$. Thus reality restricts the point, with no additional condition on $W$.
\end{proposition}

\begin{proof}
By Proposition~\ref{prop:primeclass}, the quotient $\F/\mathcal P_{x,W}$ is prime with attached point $x$. On an affine $U=\Spec A$ containing $x$, its sections are prime with annihilator $\p_x$ (Lemma~\ref{lem:primesections}); the module classification in Theorem~\ref{thm:modclass}(b) says they are real if and only if $\kappa(x)=\Frac(A/\p_x)$ is real. On affine opens not containing $x$, the quotient is zero. Proposition~\ref{prop:sheafaffine}(b) proves the equivalence, and the classification gives the last assertion.
\end{proof}

The point kernels at closed points give skyscraper quotients.

\begin{lemma}[Skyscraper quotients]\label{lem:maxsheaf}
Let $x\in X$ be closed and $W\subseteq\F(x)$ a subspace.
\begin{enumerate}
\item[(a)] The stalks of $\mathcal P_{x,W}$ are $\F_y$ for $y\neq x$, and the inverse image of $W$ in $\F_x$ for $y=x$. Moreover,
\[
\F/\mathcal P_{x,W}\cong(j_x)_*(\F(x)/W).
\]
If $\F$ is coherent, so is $\mathcal P_{x,W}$.
\item[(b)] $\mathcal P_{x,W}$ is maximal if and only if $W$ is a hyperplane.
\end{enumerate}
\end{lemma}
\begin{proof}
(a) Since $x$ is closed, the map $\F\to(j_x)_*(\F(x)/W)$ is surjective: its stalk map at $x$ is the fibre quotient, and all other target stalks vanish. This gives the quotient and stalk descriptions. Coherence follows from quasi-coherence of the kernel when $\F$ is coherent and $X$ is Noetherian.

(b) Quasi-coherent subsheaves of this skyscraper correspond to $\kappa(x)$-subspaces of $\F(x)/W$. Hence the quotient is nonzero and simple exactly when $W$ is a hyperplane.
\end{proof}

\begin{proposition}[Classification of maximal subsheaves]\label{prop:maxsheaf}
Let $\F$ be quasi-coherent. The correspondence of Proposition~\ref{prop:primeclass} restricts to a bijection between pairs $(x,H)$ with $x$ closed and $H\subsetneq\F(x)$ a hyperplane, and the maximal subsheaves of $\F$. A maximal subsheaf $\mathcal P_{x,H}$ is real if and only if $\kappa(x)$ is real.
\end{proposition}

\begin{proof}
Let $\mathcal M$ be maximal, so $\mathcal Q=\F/\mathcal M$ is nonzero and simple. Every quasi-coherent sheaf on a Noetherian scheme is a union of coherent subsheaves \cite[Tag~0GN6]{Stacks}; simplicity therefore makes $\mathcal Q$ coherent. Choose a closed point $x$ in its support. Nakayama's lemma gives $\mathcal Q(x)\neq0$, and a nonzero linear form on this fibre gives a surjection $\mathcal Q\to(j_x)_*\kappa(x)$ with quasi-coherent kernel. Simplicity makes this map an isomorphism. The composite $\F(x)\to\mathcal Q(x)\to\kappa(x)$ has a hyperplane kernel $H$, and $\mathcal M=\mathcal P_{x,H}$. Conversely, every such kernel is maximal by Lemma~\ref{lem:maxsheaf}(b). Uniqueness follows from the prime classification, and Proposition~\ref{prop:realprimesheaf} gives the reality criterion.
\end{proof}

\begin{proof}[Proof of Theorem~\ref{thmD}(b)]
Proposition~\ref{prop:primeclass} proves the prime classification, including quasi-coherence and uniqueness of the pair. Proposition~\ref{prop:realprimesheaf} restricts the point to a real residue field. Proposition~\ref{prop:maxsheaf} restricts the pair to a closed point and a hyperplane.
\end{proof}

For an integral $X$ with generic point $\eta$, $0=\mathcal P_{\eta,0}$ is prime in $\OO_X$. At a closed point $x$, $\OO_X/\mathcal P_{x,0}\cong(j_x)_*\kappa(x)$ is a simple skyscraper. More generally, at a closed point $x$, the dimension of the skyscraper quotient $\F/\mathcal P_{x,W}$ is the codimension of $W$ in $\F(x)$; it is simple exactly when this codimension is one.

For coherent quotients, geometric and stalk criteria for primeness are given in Proposition~\ref{prop:primestalk} (Section~\ref{sec:loc-sheaves}).

\begin{remark}
For $\F=\OO_X$, Definition~\ref{def:primesheaf} recovers Kanda's prime quasi-coherent ideal subsheaves \cite[Definition~9.4, Lemma~9.6 and Proposition~9.8(4)]{Kanda22}. Smith proves that the scheme-theoretic support of a coherent prime object is integral for Noetherian schemes with an ample line bundle \cite[Proposition~4.4]{Smith}. In \cite{Smith,Pap}, a ``prime submodule'' is a subobject which is itself prime, whereas here the quotient is prime; for example, $6\Z\subseteq\Z$ is prime in the former sense but not in the latter.
\end{remark}

\emph{Families of fibre conditions.} We will need the following extension of the point-kernel construction.

\begin{lemma}[Intersections of point kernels]\label{lem:kernelint}
Let $S\subseteq X$, and for each $x\in S$ let $W_x\subseteq\F(x)$ be a subspace. Then $\mathcal K:=\bigcap_{x\in S}\mathcal P_{x,W_x}$ is quasi-coherent, and on every affine open $U=\Spec A$, with $M=\Gamma(U,\F)$,
\[
\Gamma(U,\mathcal K)=\bigcap_{\p\in S\cap U}P_{\p,W_\p}.
\]
The intersection over an empty indexing set is understood to be $M$.
\end{lemma}

\begin{proof}
The section formula follows from Proposition~\ref{prop:primeclass}(a). Let $f\in A$. Applied to $D(f)=\Spec A_f$, it says that $\mu\in M_f$ lies in $\Gamma(D(f),\mathcal K)$ if and only if $\mu(\p)\in W_\p$ for every $\p\in S$ with $f\notin\p$. If $\mu=m/f^r$ has this property, then $fm\in\Gamma(U,\mathcal K)$: for $\p\in S$ with $f\notin\p$, $(fm)(\p)$ is a nonzero multiple of $\mu(\p)$, and for $f\in\p$, $(fm)(\p)=0$. Hence $\mu=(fm)/f^{r+1}\in\Gamma(U,\mathcal K)_f$, and the reverse inclusion is clear. Thus $\Gamma(D(f),\mathcal K)=\Gamma(U,\mathcal K)_f$ for all $f$, so $\mathcal K|_U$ is the quasi-coherent sheaf associated to $\Gamma(U,\mathcal K)$.
\end{proof}

\section{Jacobson properties}\label{sec:jacobson}

The passage from arbitrary points to closed points rests on separation by maximal quotients. The same argument gives the ordinary and real Jacobson properties; the real case additionally needs a supply of real maximal ideals.

\subsection{Separation}

The properness of a fibre kernel is controlled by the annihilator when its quotient is finitely generated.

\begin{lemma}\label{lem:Kproper}
Let $N$ be a submodule of an $R$-module $M$ and $\p$ a prime ideal of $R$.
\begin{enumerate}
\item[(a)] If $\p\not\supseteq(N:M)$, then $K(N,\p)=M$.
\item[(b)] If $M/N$ is finitely generated, then $K(N,\p)\neq M$ if and only if $\p\supseteq(N:M)$.
\end{enumerate}
\end{lemma}

\begin{proof}
(a) Any $s\in(N:M)\setminus\p$ satisfies $sm\in N$ for all $m\in M$.

(b) By (a), it remains to show that $K(N,\p)=M$ implies $\p\not\supseteq(N:M)$. Put $Q:=M/N$ and choose generators $q_1,\dots,q_n$ of $Q$. If $K(N,\p)=M$, there are $s_i\notin\p$ with $s_iq_i\in\p Q$; their product $s\notin\p$ satisfies $sQ\subseteq\p Q$. By the determinant trick \cite[Proposition~2.4]{AM}, applied to multiplication by $s$ on $Q$, there are $a_1,\dots,a_n\in\p$ such that $c:=s^n+a_1s^{n-1}+\dots+a_n$ annihilates $Q$. Then $c\in(N:M)$, and $c\notin\p$, since $c\equiv s^n\not\equiv0$ modulo $\p$.
\end{proof}

An $R$-module $Q$ is torsionless if every nonzero $q\in Q$ is detected by a homomorphism $Q\to R$, equivalently if $Q$ embeds into a product $R^I$ \cite[\S4H, Definition~4.64 and Remark~4.65(a)]{Lam}. Over a domain, torsionless modules are torsion-free \cite[Remark~4.65(d)]{Lam}. Finite generation gives the converse needed below.

\begin{lemma}[{\cite[Remark~4.65(e)]{Lam}}]\label{lem:embed}
Let $D$ be a domain and $Q$ a finitely generated torsion-free $D$-module. Then $Q$ is isomorphic to a submodule of $D^r$ for some $r\ge0$; in particular, $Q$ is torsionless.
\end{lemma}

\begin{proof}
Put $K=\Frac D$. Torsion-freeness embeds $Q$ into the finite-dimensional space $Q\otimes_D K\cong K^r$. Choose $0\neq a\in D$ clearing the denominators of the coordinates of a finite generating set of $Q$. Then multiplication by $a$ embeds $Q$ into $D^r$.
\end{proof}

\begin{theorem}\label{thm:jac}
Let $P$ be a prime submodule of an $R$-module $M$, $\p:=(P:M)$, and let $\M$ be a set of maximal ideals of $R$ containing $\p$ such that $\p=\bigcap_{\m\in\M}\m$. If $M/P$ is a torsionless $R/\p$-module (e.g.\ if $M/P$ is finitely generated as an $R$-module, by Lemmas~\ref{lem:primequot} and~\ref{lem:embed}), then
\[
P=\bigcap_{\m\in\M}(P+\m M).
\]
Consequently, $P$ is an intersection of maximal submodules $P'$ with $(P':M)\in\M$.
\end{theorem}

\begin{proof}
Put $D:=R/\p$ and $Q:=M/P$. Let $x\in M\setminus P$ with image $q\neq0$ in $Q$. Choose $\varphi:Q\to D$ with $\varphi(q)\neq0$, and $c\in R$ with $\bar c=\varphi(q)$; then $c\notin\p$, so there is $\m\in\M$ with $c\notin\m$. Since $\varphi(\m Q)=\m\varphi(Q)\subseteq\m D=\m/\p$ and $\varphi(q)=\bar c\notin\m/\p$, we get $q\notin\m Q$, i.e.\ $x\notin P+\m M$. Hence $P=\bigcap_{\m\in\M}(P+\m M)$. For any $\m$, linear functionals on the $R/\m$-vector space $M/(P+\m M)$ separate its nonzero vectors. Their nonzero composites $M\to R/\m$ are surjective, with maximal kernels of annihilator $\m$ by Lemma~\ref{lem:simple}. Thus each proper $P+\m M$ is an intersection of such kernels, proving the consequence.
\end{proof}

\begin{remark}[Vanishing fibres]\label{rem:empty}
In Theorem~\ref{thm:jac}, $P+\m M=M$ can occur, equivalently $\m(M/P)=M/P$ or $K(P,\m)=M$ by Lemma~\ref{lem:fibreid}(d). If $M/P$ is finitely generated, this never happens, by Lemma~\ref{lem:Kproper}(b), since $\m\supseteq\p=(P:M)$. The same holds if $R$ is Noetherian. Indeed, let $D:=R/\p$, $\bar\m:=\m/\p$, and, using torsionlessness, identify $M/P$ with a $D$-submodule of $D^I$. If $\bar\m(M/P)=M/P$, then the image $J_i$ of $M/P$ under the $i$-th projection is an ideal of $D$ with $\bar\m J_i=J_i$; as $J_i$ is finitely generated, $(1+a)J_i=0$ for some $a\in\bar\m$ \cite[Corollary~2.5]{AM}, and $1+a\neq0$; since $D$ is a domain, $J_i=0$. As this holds for all $i$, $M/P=0$, a contradiction.

Without these hypotheses, $P+\m M=M$ can occur. Let $A:=\R[x^{1/2^n}:n\ge0]$, viewed as a ring of functions on $[0,\infty)$, and for $t\ge0$ let $\m_t$ be the kernel of the evaluation $x^{1/2^n}\mapsto t^{1/2^n}$. The $\m_t$ are real maximal ideals of $A$ with $\bigcap_{t\ge0}\m_t=0$, as a nonzero polynomial in $x^{1/2^n}$ cannot vanish on $[0,\infty)$. Apply Theorem~\ref{thm:jac} with $R:=A$, $\M:=\{\m_t:t\ge0\}$, $M:=\m_0$ and $P:=0$. Since $A$ is a domain and $\m_0\neq0$, the submodule $P$ is prime with $(P:M)=0$, and $M/P=\m_0$ is torsionless over $A/(P:M)=A$, being an $A$-submodule of $A$. But $\m_0=\m_0^2$. Indeed, every $f\in\m_0$ lies in $\R[y]$ for some $y=x^{1/2^N}$ and satisfies $f(0)=0$, so $f=yg$ with $g\in\R[y]$; with $z:=x^{1/2^{N+1}}$ we get $f=z\cdot(zg)\in\m_0^2$, as $z,zg\in\m_0$. Hence $P+\m_0M=M$. As all $\m_t$ are real, this also happens for real maximal ideals.
\end{remark}

\begin{example}\label{ex:Q}
(a) Over the Jacobson ring $\Z$, the module $\Q$ is torsion-free but not torsionless \cite[Remark~4.65(e)]{Lam}, so $0$ is a prime submodule; but $\Q$ has no maximal submodule, since $\Q/P\cong\Z/p$ would force $P\supseteq p\Q=\Q$ (cf.\ \cite[Examples~2.4(ii)]{MSS}).

(b) \emph{Freeness does not suffice.} Let $\pi:\Z^{(\N)}\to\Q$ be a surjection and $P:=\ker\pi$. Then $\Z^{(\N)}/P\cong\Q$ is a prime module, so $P$ is a prime submodule of a free module over a Jacobson ring, but by (a) no maximal submodule contains $P$. Thus freeness of $M$ cannot replace torsionlessness of $M/P$ in Theorem~\ref{thmA}(a). (Conversely, if $R$ is not a Jacobson ring, then no nonzero free module is Jacobson \cite[Examples~2.4(i)]{MSS}.)

(c) \emph{A torsionless module which is not finitely generated.} The $\Z$-module $\Z^\N$ is torsion-free and torsionless (by \cite[Remark~4.65(c)]{Lam}), so $0$ is a prime submodule and an intersection of maximal submodules by Theorem~\ref{thm:jac}. It is not a submodule of a free module, since subgroups of free abelian groups are free and $\Z^\N$ is not free by the Baer--Specker theorem \cite[Chapter~3, Theorem~8.2]{Fuchs}; so the embedding argument of \cite[Lemma~2.8]{MSS} does not apply to it.

(d) \emph{Torsionlessness is not necessary.} Let $p$ be a prime and $Q:=\Z[1/p]$. Then $0$ is a prime submodule and $Q/qQ\cong\Z/q$ for all primes $q\neq p$, while $pQ=Q$; so $0=\bigcap_{q\neq p}qQ$ is an intersection of maximal submodules. However, $\Hom_\Z(Q,\Z)=0$ (the image of $1$ under a homomorphism $\varphi$ is divisible by all powers of $p$, so $\varphi(1)=0$, and then $p^k\varphi(p^{-k})=0$, so $\varphi=0$, as $\Z$ is torsion-free), so $Q$ is not torsionless.

(e) \emph{A real torsion-free module with vanishing closed fibres.} Torsionlessness cannot be weakened to torsion-freeness in Theorem~\ref{thmA}(b), and finite generation cannot be dropped in Theorem~\ref{thmB}(b). Let $R=\R[x]$ and $M=\R(x)$. Then $M$ is a real module: multiplication $M\times M\to\R(x)$, $(v,w)\mapsto vw$, is $R$-bilinear, so it induces an $R$-linear map $\mu:M\otimes_RM\to\R(x)$ with $\mu(\sum_im_i\otimes m_i)=\sum_im_i^2$. If $\sum_im_i\otimes m_i=0$, then $\sum_im_i^2=0$ in the real field $\R(x)$, so all $m_i=0$. Thus $\rrad0=0$, and $M$ is semireal, as $M\neq0$. But $M$ has no maximal submodules at all, since $\m M=M$ for every maximal ideal $\m$ of $R$; in particular $M(a)=0$ for every $a\in\R$. Thus both the closed-fibre radical formula and the semireality test following Theorem~\ref{thm:modfibre} fail without finite generation. As in (b), the kernel of a surjection from a free $R$-module onto $\R(x)$ is a real prime submodule of a free module which is contained in no maximal submodule.
\end{example}

\subsection{Ordinary and real Jacobson properties for modules}

Recall that quotients of Jacobson rings are Jacobson and that finitely generated algebras over a field or $\Z$ are Jacobson \cite[Chapter~4]{Eis}. Quotients of real Jacobson rings are real Jacobson by the correspondence of real ideals in Section~\ref{sec:basic}. In the finite type case, the required real maximal ideals are supplied by Lang's theorem; importantly, the base field need not be real closed.

\begin{proposition}\label{prop:fgreal}
Let $A$ be a finitely generated algebra over a real field $k$. Then every real prime ideal of $A$ is an intersection of real maximal ideals.
\end{proposition}

\begin{proof}
Let $\p$ be a real prime ideal and $g\notin\p$. Set $D:=A/\p$ and order $\operatorname{Frac}D$. Lang's theorem \cite[Chapter~XI, \S3, Theorem~3.1, p.~457]{Lang}, applied to the finitely generated $k$-algebra $D[1/\bar g]$, gives a $k$-homomorphism
\[
\varphi:D[1/\bar g]\longrightarrow k^{\mathrm{rc}},
\]
where $k^{\mathrm{rc}}$ is a real closure for the ordering induced on $k$; see also \cite[p.~111]{Dubois} and \cite[II.2, (2.4), p.~52]{FG}. The image $\varphi(D)$ is a finitely generated algebraic $k$-algebra and a domain, hence a field; it is real, being contained in $k^{\mathrm{rc}}$. Thus the kernel of $A\to D\to\varphi(D)$ is a real maximal ideal containing $\p$ but not $g$, since $\varphi(\bar g)$ is invertible. Intersecting these ideals gives $\p$.
\end{proof}

\begin{remark}\label{rem:fgreal}
Let $A$ be a finitely generated algebra over a field $k$.

(a) If $k$ is not real, then $A$ has no real prime ideals: if $-1=\sum_ic_i^2$ in $k$, then $1^2+\sum_ic_i^2=0$ lies in every ideal, so a real ideal contains $1$. Together with Proposition~\ref{prop:fgreal}, every finitely generated algebra over a field is real Jacobson.

(b) If $k$ is real closed, the real maximal ideals of $A$ are the $\m_a$ with $a\in V_A(k)$: by Zariski's lemma \cite[Proposition~7.9]{AM} each residue field $A/\m$ is a finite extension of $k$, and a real closed field has no proper real algebraic extension \cite[I.1, Theorem~1.29]{FG}; see also \cite[II.3, Theorem~3.4]{FG}.
\end{remark}

By the decomposition theorems \cite[Theorems~A and~B]{Cdec}, the ordinary, resp.\ real, Jacobson property is equivalent to requiring every semiprime, resp.\ real, submodule to be an intersection of maximal, resp.\ real maximal, submodules.

\begin{proof}[Proof of Theorem~\textup{\ref{thmA}(a),(b)}]
Put $\p=(P:M)$ and take all maximal ideals containing $\p$ in (a), or all real maximal ideals containing $\p$ in (b). The ring hypothesis says their intersection is $\p$; in (b), $\p$ is real by Lemma~\ref{lem:primequot}. Theorem~\ref{thm:jac} expresses $P$ as an intersection of maximal submodules with annihilators in this family, which are real in (b) by Theorem~\ref{thm:modclass}(c). For finitely generated prime quotients, Lemmas~\ref{lem:primequot} and~\ref{lem:embed} supply torsionlessness. Hence the corresponding ring property implies that every finitely generated module has the corresponding Jacobson property; the converse follows by taking $M=R$.
\end{proof}

The real Jacobson hypothesis in (b) cannot be replaced by the ordinary Jacobson property: $\Z$ is Jacobson and $0$ is a real prime ideal, but no maximal ideal $p\Z$ is real, since $-1$ is a sum of squares in $\mathbb F_p$; so $0\subseteq\Z$ is a real prime submodule that is not an intersection of real maximal submodules.

A module which is not semireal has no real prime submodules, so its real Jacobson property is vacuous. We do not require real Jacobson rings or modules to be semireal; in particular, this keeps real Jacobson rings closed under quotients. For example, $\R[x]/(x^2+1)$ is not semireal, although $\R[x]$ is real Jacobson.

\begin{corollary}\label{cor:maxreal}
Let $R$ be a real Jacobson ring and $M$ a finitely generated $R$-module. Then every submodule of $M$ which is maximal among the proper real submodules is a real maximal submodule.
\end{corollary}

\begin{proof}
If $N$ is maximal among the proper real submodules, then $N$ is an intersection of a nonempty family of real maximal submodules by \cite[Theorem~B]{Cdec} and Theorem~\ref{thmA}(b), and each of them equals $N$ by maximality. This refines the statement of \cite[Theorem~B]{Cdec} that such $N$ are prime submodules.
\end{proof}

\begin{proposition}[Annihilator criteria]\label{prop:ann}
Let $M$ be a finitely generated $R$-module with annihilator $I$.
\begin{enumerate}
\item[(a)] $M$ is Jacobson if and only if $R/I$ is a Jacobson ring.
\item[(b)] $M$ is real Jacobson if and only if $R/I$ is a real Jacobson ring.
\end{enumerate}
\end{proposition}
\begin{proof}
In (a) use all prime and maximal ideals and submodules; in (b) use only the real ones. If $R/I$ has the indicated property, the annihilator of each prime submodule $P$ contains $I$ and is an intersection of the indicated maximal ideals. Theorem~\ref{thm:jac} and Lemma~\ref{lem:embed} express $P$ as an intersection of the corresponding maximal submodules.

Conversely, let $\p\supseteq I$ be a prime ideal of the indicated type. Lemma~\ref{lem:Kproper}(b) makes $K(0,\p)$ proper, and Theorem~\ref{thm:modclass} makes it prime with annihilator $\p$, and real in (b). Write it as an intersection of the corresponding maximal submodules. Taking annihilators expresses $\p$ as an intersection of the corresponding maximal ideals. The quotient correspondence for ideals (Section~\ref{sec:basic}) proves the ring property.
\end{proof}

\subsection{Separation for sheaves}\label{sec:jacobson-sheaves}

The module separation theorem passes to sheaves through the classification of point kernels. Here coherence is needed for the prime quotient, rather than for classification of the maximal subsheaves.

\begin{proposition}[Separation by maximal subsheaves]\label{prop:sheafseparation}
Let $X$ be of finite type over a field and $\F$ quasi-coherent. Every prime subsheaf $\mathcal P\subsetneq\F$ with coherent quotient is an intersection of maximal quasi-coherent subsheaves containing it. If $\mathcal P$ is real prime, it is an intersection of real maximal subsheaves.
\end{proposition}
\begin{proof}
Write $\mathcal P=\mathcal P_{\eta,W}$ by Proposition~\ref{prop:primeclass}. On an affine $U=\Spec A$ containing $\eta$, its sections $P\subseteq M=\Gamma(U,\F)$ are prime with finitely generated quotient, hence torsionless over $A/(P:M)$ by Lemmas~\ref{lem:primequot} and~\ref{lem:embed}. The ring $A$ is Jacobson, and real Jacobson by Proposition~\ref{prop:fgreal} and Remark~\ref{rem:fgreal}(a). In the real case, $P$ is real by Proposition~\ref{prop:sheafaffine}(b). Apply Theorem~\ref{thm:jac} to all maximal ideals containing $(P:M)$, or all real maximal ideals in the real case. Thus $P$ is an intersection of the corresponding maximal submodules of $M$.

By the classification, these are precisely the $P_{\m,H}$ with $P(\m)\subseteq H$ a hyperplane, restricting to real residue fields in the real case. Their residue fields are finite over $k$ by Zariski's lemma, so their points are closed in $X$ \cite[Tag~01TF, applied to $X\to\Spec k$]{Stacks}. They therefore define global maximal subsheaves containing $\mathcal P$. An affine chart not containing $\eta$ misses its closure and has $\mathcal P|_U=\F|_U$; no closed point there imposes a condition. Lemma~\ref{lem:kernelint}, grouping hyperplanes over each point, assembles the affine intersections into the required global intersection.

This proves Theorem~\ref{thmA}(c).
\end{proof}

\section{Fibre Nullstellens\"atze for modules}\label{sec:fibres}

A set of points suffices for a radical exactly when it suffices to separate the corresponding primes. We first formulate this principle for arbitrary fibre conditions.

For $S\subseteq\Spec R$ and subspaces $W_\p\subseteq M(\p)$, put
\[
\Phi_S(N):=(N(\p))_{\p\in S},\qquad
\Psi_S(W):=\bigcap_{\p\in S}P_{\p,W_\p},\qquad
\rad_S:=\Psi_S\Phi_S.
\]
Families are ordered componentwise. Since $N(\p)$ is spanned by the values of elements of $N$,
\begin{equation}\label{eq:fibreadjunction}
\Phi_S(N)\subseteq W\quad\Longleftrightarrow\quad N\subseteq\Psi_S(W).
\end{equation}
Both maps preserve inclusions. This Galois connection gives $\Phi\Psi\Phi=\Phi$ and $\Psi\Phi\Psi=\Psi$: hence $\rad_S N$ is the largest submodule with the same fibre images as $N$ on $S$. The maps restrict to inverse inclusion-preserving bijections between submodules fixed by $\rad_S$ and \emph{realizable families}, meaning families of the form $\Phi_S(N)$. Equivalently, a family is realizable if $\Phi\Psi(W)=W$.

Write $\Spec_{\mathrm{re}}R:=\{\p\in\Spec R:\kappa(\p)\text{ is real}\}$ and $\operatorname{Max}_{\mathrm{re}}(R):=\operatorname{Max}(R)\cap\Spec_{\mathrm{re}}R$.

\begin{theorem}[Fibre Nullstellens\"atze]\label{thm:modfibre}
Let $M$ be an $R$-module and $S\subseteq\Spec R$.
\begin{enumerate}
\item[(a)] For every $N\subseteq M$,
\[
\rad_S N=\bigcap_{\substack{P\supseteq N,\ P\text{ prime}\\(P:M)\in S}}P.
\]
\item[(b)] $S=\Spec R$ gives $\rad_S=\rad$, and $S=\Spec_{\mathrm{re}}R$ gives $\rad_S=\rrad$. If $M$ is finitely generated, $S=\operatorname{Max}(R)$ gives $\rad_S=\rad$ when $R$ is Jacobson, and $S=\operatorname{Max}_{\mathrm{re}}(R)$ gives $\rad_S=\rrad$ when $R$ is real Jacobson.
\end{enumerate}
\end{theorem}

\begin{proof}
(a) By definition,
\[
\rad_S N=\bigcap_{\p\in S}P_{\p,N(\p)}.
\]
Fix $\p\in S$. Theorem~\ref{thm:modclass}(a) identifies the primes containing $N$ with annihilator $\p$ as the $P_{\p,W}$ with $N(\p)\subseteq W\subsetneq M(\p)$. If $N(\p)$ is proper, $P_{\p,N(\p)}$ is the smallest of these primes, so it is their intersection. If $N(\p)=M(\p)$, there are no such primes; their empty intersection and $P_{\p,N(\p)}$ both equal $M$. Intersecting over $\p\in S$ proves (a) and also gives
\[
\rad_S N=\bigcap_{\p\in S}K(N,\p).
\]
Thus $\rad_S N$ is semiprime and contains $\rad N$. If $S\subseteq\Spec_{\mathrm{re}}R$, all the primes in (a) are real by Theorem~\ref{thm:modclass}(b), so $\rad_S N$ is real and contains $\rrad N$.

(b) The first two equalities follow from (a) and the prime-decomposition theorems \cite[Theorems~A and~B]{Cdec}. Suppose now that $M$ is finitely generated and $R$ is Jacobson. By \cite[Theorem~A]{Cdec}, $\rad N$ is an intersection of prime submodules containing $N$, and Theorem~\ref{thmA}(a) expresses each of these as an intersection of maximal submodules. Each such maximal submodule is prime with maximal annihilator by Theorem~\ref{thm:modclass}(c), so it occurs in the intersection defining $\rad_{\operatorname{Max}(R)}N$ in (a). Hence $\rad_{\operatorname{Max}(R)}N\subseteq\rad N$, while (a) gives the reverse inclusion. The same argument with \cite[Theorem~B]{Cdec} and Theorem~\ref{thmA}(b) proves the real equality over a real Jacobson ring. In these closed cases the radical is also the intersection of the maximal, resp.\ real maximal, submodules containing $N$.

If $R$ is a finitely generated $k$-algebra, Zariski's lemma \cite[Proposition~7.9]{AM} identifies all maximal ideals with the $\m_a$, $a\in V_R(k)$, when $k$ is algebraically closed; Remark~\ref{rem:fgreal}(b) identifies the real maximal ideals with these $\m_a$ when $k$ is real closed. Thus the corresponding closed-fibre conditions read $f(a)\in N(a)$ for every $a\in V_R(k)$. In each case, the Galois bijection identifies semiprime, resp.\ real, submodules with realizable families on the indicated points.
\end{proof}

\begin{proof}[Proof of Theorem~\textup{\ref{thmB}(a),(b)}]
Take $S=\operatorname{Max}(R)$ in (a) and $S=\operatorname{Max}_{\mathrm{re}}(R)$ in (b). By Theorem~\ref{thm:modfibre}(b), $\rad=\rad_S$, resp.\ $\rrad=\rad_S$, and by definition $\rad_S N=\{f\in M:f(\m)\in N(\m)\text{ for all }\m\in S\}$; this is the membership criterion. The maximality and bijection statements are the Galois identities for this $S$.
\end{proof}

\begin{example}[Polynomial radial vector fields]\label{ex:radial}
Consider the module of \cite[Theorem~11]{JS}: let $k$ be a real closed field, $R=k[x,y]$, $f=(x,y)\in R^2$ and $N=(x,y)f\subseteq R^2$. Then $N(a)=k\,a$ for $0\neq a\in k^2$ and $N(0)=0$, so by Theorem~\ref{thmB}(b), $g\in\rrad N$ if and only if $g(a)\in k\,a$ for every $a\in k^2$; as $k$ is infinite, this means $g_1y=g_2x$, i.e.\ $g\in Rf$. Since $f\in(N:f)R^2$, semiprimeness forces $f\in\rad N\subseteq\rrad N$, so
\[
\rad N=\rrad N=Rf,
\]
the polynomial radial vector fields, although $f\notin N$. Moreover, $Rf=P_{(0),\,\kappa((0))f}$ with $\kappa((0))=k(x,y)$ is a prime submodule attached to the generic point.
\end{example}

\begin{lemma}[Properness and semireality]\label{lem:modproper}
Let $N\subseteq M$, $Q=M/N$ and $S=\Spec_{\mathrm{re}}R$. The following conditions are equivalent:
\begin{center}
\begin{tabular}{@{}p{0.30\linewidth}@{\hspace{0.04\linewidth}}p{0.65\linewidth}@{}}
(i) $Q$ is semireal;  & (iv) there is a real prime submodule $P\subset M$ containing $N$; \\[3pt]
(ii) $Q(\p)\neq0$ for some $\p\in S$; & (v) $N(\p)\neq M(\p)$ for some $\p\in S$;  \\[3pt]
(iii) $\rad_S 0\neq Q$;  & (vi) $\rad_S N\neq M$;
\end{tabular}
\end{center}
If $Q$ is finitely generated, these are also equivalent to
\begin{center}
\begin{tabular}{@{}p{0.30\linewidth}@{\hspace{0.04\linewidth}}p{0.65\linewidth}@{}}
(vii)  $R/\Ann_R(Q)$ is semireal; & (viii) $(N:M)$  is contained in a real prime ideal.
\end{tabular}
\end{center}
If $Q$ is finitely generated and $R$ is real Jacobson, one may take $S=\operatorname{Max}_{\mathrm{re}}(R)$ instead. For this choice of $S$, ``prime'' in (iv) may be replaced by ``maximal''.
\end{lemma}
\begin{proof}
For any $S\subseteq\Spec R$, right exactness of tensoring with $\kappa(\p)$ gives $Q(\p)=M(\p)/N(\p)$; thus
\[
\rad_S N\neq M
\quad\Longleftrightarrow\quad
Q(\p)\neq0\text{ for some }\p\in S
\quad\Longleftrightarrow\quad
N(\p)\neq M(\p)\text{ for some }\p\in S.
\]
Indeed, the last condition means that some $f\in M$ fails a fibre condition defining $\rad_S N$. Thus (ii), (v) and (vi) are equivalent. Since $S=\Spec_{\mathrm{re}}R$, Theorem~\ref{thm:modfibre}(b) gives $\rad_S0=\rrad_Q0$. Semireality of $Q$ means that it has a nonzero real quotient, equivalently $\rrad_Q0\neq Q$, proving (i)$\Leftrightarrow$(iii). Applying the display to $0\subseteq Q$ gives (ii)$\Leftrightarrow$(iii). By Theorem~\ref{thm:modfibre}(a) and Theorem~\ref{thm:modclass}(b), $\rad_S N$ is the intersection of the real prime submodules containing $N$; this intersection is proper exactly when there is such a submodule. Hence (iv)$\Leftrightarrow$(vi).

Since $\Ann_R(Q)=(N:M)$, condition (vii) is equivalent to (viii) by the real prime criterion for semireal rings. If $Q$ is finitely generated, Lemma~\ref{lem:Kproper}(b) gives $Q(\p)\neq0$ exactly when $(N:M)\subseteq\p$, proving (ii)$\Leftrightarrow$(viii). If $R$ is also real Jacobson, every real prime is contained in a real maximal ideal; applying the same lemma shows that closed real fibres suffice. For $S=\operatorname{Max}_{\mathrm{re}}(R)$, condition (v) gives $N(\m)\subsetneq M(\m)$ for some $\m\in S$. Choose a hyperplane $H$ containing $N(\m)$; then $P_{\m,H}$ is a real maximal submodule containing $N$, by Theorem~\ref{thm:modclass}(c). Conversely, a real maximal submodule containing $N$ gives such a proper fibre, and is real prime. A nonzero finitely generated $Q$ also has a nonzero closed fibre over any ring: choose a maximal ideal containing its annihilator. These observations prove the proper-fibre assertions following Theorem~\ref{thmB}.

\end{proof}

Without finite generation, semireality of $Q$ still implies (viii), since its annihilator lies in that of a real prime quotient. The converse fails for the faithful $\Z$-module $\bigoplus_{p\text{ prime}}\Z/p\Z$: every nonzero quotient is torsion, whereas real $\Z$-modules are torsion-free \cite[Example~4.5(3)]{Cdec}. Closed fibres also need not detect semireality: $\R(x)$ is real over $\R[x]$ but all its closed fibres vanish (Example~\ref{ex:Q}(e)). Parts (a) and (e) of Example~\ref{ex:Q} likewise show that finite generation is needed in Proposition~\ref{prop:ann}.

\begin{example}[Semiprime and semireal, but not real]\label{ex:realex}
Let $R=\R[x]$, $M=R^2$ and $N:=(x^2+1)R\oplus0$. Then $M/N\cong\R[x]/(x^2+1)\oplus R$ is a direct sum of semiprime modules, so $N$ is a semiprime submodule \cite[Lemma~2.1(f)]{Cdec}. But $N(a)=\R\oplus0$ for every $a\in\R$, as $a^2+1\neq0$, so $\rrad N=R\oplus0\supsetneq N$ by Theorem~\ref{thm:modfibre}(b); hence $N$ is not a real submodule. Thus the module $M/N$ is semiprime and semireal, as it has the nonzero real quotient $R$, but not real. Hence semiprimeness and semireality together do not imply reality.
\end{example}

\begin{remark}[Directional vanishing]\label{sec:directional}
A \emph{directional point} of $M$ is a pair $(\p,H)$ with $H\subset M(\p)$ a hyperplane; write $\operatorname{Dir}M$ for their set. The corresponding $P_{\p,H}$ are the fibrewise maximal primes (compare \cite[Corollaries~9 and~11]{THS}). For $S\subseteq\Spec R$, define
\[
V_S(N):=\{(\p,H)\in\operatorname{Dir}M:\p\in S,\ N(\p)\subseteq H\},
\qquad
J(T):=\bigcap_{(\p,H)\in T}P_{\p,H},
\]
where $T$ is a set of directional points over $S$.

Every fibre subspace is the intersection of the hyperplanes containing it, with empty intersection $M(\p)$. Thus $V_S(N)$ carries exactly the same information as $\Phi_S(N)$, and taking inverse images gives
\begin{equation}\label{eq:radSdir}
J(V_S(N))=\Psi_S\Phi_S(N)=\rad_S N.
\end{equation}
The $\Phi_S$--$\Psi_S$ correspondence therefore becomes an inclusion-reversing bijection between submodules fixed by $\rad_S$ and directional zero sets. The former are the semiprime, resp.\ real, submodules for the corresponding choices of $S$ in Theorem~\ref{thm:modfibre}.

For $M=R^n$, a hyperplane is the kernel of a nonzero form $\lambda_v(f(\p))=\sum_i v_i f_i(\p)$, so directional vanishing recovers the formalism of \cite{C22a,C13}. For $M=R$, there is one direction over each point, giving the usual ideal correspondence.
\end{remark}

\section{Fibre Nullstellens\"atze for sheaves}\label{sec:fibres-sheaves}

For quasi-coherent $\G\subseteq\F$ on a Noetherian scheme $X$, $S\subseteq X$ and subspaces $W_x\subseteq\F(x)$, define
\[
\Phi_S(\G):=(\G(x))_{x\in S},\qquad
\Psi_S(W):=\bigcap_{x\in S}\mathcal P_{x,W_x},\qquad
\rad_S:=\Psi_S\Phi_S.
\]
Lemma~\ref{lem:kernelint} makes $\Psi_S(W)$ quasi-coherent. Proposition~\ref{prop:primeclass}(a) gives the Galois connection \eqref{eq:fibreadjunction}; its identities, largest-subobject interpretation and bijection with realizable families therefore apply unchanged.

Write $X_{\mathrm{re}}:=\{x\in X:\kappa(x)\text{ is real}\}$ and let $X_{\mathrm{cl}}$ be the set of closed points of $X$.

\begin{theorem}[Fibre Nullstellens\"atze for sheaves]\label{thm:sheaffibre}
Let $\F$ be quasi-coherent on a Noetherian scheme $X$ and $S\subseteq X$.
\begin{enumerate}
\item[(a)] For every quasi-coherent $\G\subseteq\F$,
\[
\rad_S\G=\bigcap_{\substack{\mathcal P\supseteq\G,\ \mathcal P\text{ prime}\\\text{attached point in }S}}\mathcal P.
\]
\item[(b)] $S=X$ gives $\rad_S=\rad$, and $S=X_{\mathrm{re}}$ gives $\rad_S=\rrad$. If $\F$ is coherent and $X$ is of finite type over a field, $S=X_{\mathrm{cl}}$ gives $\rad_S=\rad$, and $S=X_{\mathrm{re}}\cap X_{\mathrm{cl}}$ gives $\rad_S=\rrad$.
\end{enumerate}
\end{theorem}

\begin{proof}
(a) By definition,
\[
\rad_S\G=\bigcap_{x\in S}\mathcal P_{x,\G(x)}.
\]
Fix $x\in S$. Proposition~\ref{prop:primeclass} identifies the primes containing $\G$ and attached to $x$ as the $\mathcal P_{x,W}$ with $\G(x)\subseteq W\subsetneq\F(x)$. If $\G(x)$ is proper, $\mathcal P_{x,\G(x)}$ is the smallest of these primes, so it is their intersection. If $\G(x)=\F(x)$, there are no such primes; their empty intersection and $\mathcal P_{x,\G(x)}$ both equal $\F$. Intersecting over $x\in S$ proves (a), as in the module argument of Theorem~\ref{thm:modfibre}(a). More explicitly, on $U=\Spec A$, put $M=\Gamma(U,\F)$ and $N=\Gamma(U,\G)$, and identify $S\cap U$ with a subset of $\Spec A$. Lemma~\ref{lem:kernelint} makes $\rad_S\G$ quasi-coherent with affine sections $\rad_{S\cap U}N$. By Theorem~\ref{thm:modfibre}(a), these submodules are semiprime, and real if $S\subseteq X_{\mathrm{re}}$. The affine criteria of Proposition~\ref{prop:sheafaffine}(a),(b) therefore make $\rad_S\G$ semiprime, resp.\ real, so it contains $\rad\G$, resp.\ $\rrad\G$.

(b) For $S=X$ or $S=X_{\mathrm{re}}$, Lemma~\ref{lem:kernelint} and Theorem~\ref{thm:modfibre}(b) identify the affine sections with the corresponding module radical. Thus, on $U=\Spec A$, with $M=\Gamma(U,\F)$ and $N=\Gamma(U,\G)$,
\begin{equation}\label{eq:sheafaffineradicals}
(\rad\G)|_U=\widetilde{\rad N},\qquad
(\rrad\G)|_U=\widetilde{\rrad N}.
\end{equation}
By Proposition~\ref{prop:sheafaffine}, these intersections are semiprime, respectively real. Every such quasi-coherent supersheaf $\mathcal H\supseteq\G$ has affine sections containing $\rad N$, respectively $\rrad N$, by the same proposition. Hence $\mathcal H$ contains the corresponding intersection, proving minimality. They are coherent when $\F$ is coherent, since $X$ is Noetherian. Together with (a), this proves the unrestricted prime-intersection decompositions.

Suppose now that $\F$ is coherent and $X$ is of finite type over a field. Every prime quotient of $\F$ is coherent, so Theorem~\ref{thmA}(c) expresses each prime subsheaf as an intersection of maximal subsheaves. By Proposition~\ref{prop:maxsheaf}, these are primes attached to closed points, hence they occur in (a) for $S=X_{\mathrm{cl}}$. Substituting these intersections in the unrestricted prime decomposition gives $\rad_{X_{\mathrm{cl}}}\G\subseteq\rad\G$; (a) gives the reverse inclusion. The same argument with real primes and real maximal subsheaves gives $\rad_{X_{\mathrm{re}}\cap X_{\mathrm{cl}}}\G=\rrad\G$. In these closed cases the radicals are also the intersections of the maximal, resp.\ real maximal, subsheaves containing $\G$. Over a real closed base field, the closed real points are precisely the rational points by Remark~\ref{rem:fgreal}(b).

In each case, the Galois identities identify semiprime, resp.\ real, subsheaves with realizable families on the indicated points. The radical is the largest quasi-coherent subsheaf with the same fibre images as $\G$ there; its sections over an open $U$ are the $s$ with $s(x)\in\G(x)$ for all $x\in S\cap U$, by Proposition~\ref{prop:primeclass}(a).
\end{proof}

\begin{remark}[Ideal case]\label{rem:idealcase}
For $\F=\OO_X$, each fibre $\kappa(x)$ has only the subspaces $0$ and $\kappa(x)$, and $\mathcal I(x)=0$ exactly when $x\in V(\mathcal I)$. Hence Theorem~\ref{thm:sheaffibre} gives
\[
\Gamma(U,\rad\mathcal I)=\{s\in\OO_X(U):s(x)=0\text{ for all }x\in U\cap V(\mathcal I)\},
\]
which coincides with the ideal sheaf of the reduced induced structure on $V(\mathcal I)$ \cite[Tag~01J3]{Stacks}. Thus radical quasi-coherent ideal sheaves correspond to closed subsets of $X$ \cite[Proposition~5.2.1]{EGA1}; closed points suffice on schemes of finite type over a field. Likewise, $\rrad\mathcal I$ is the ideal of sections vanishing on $V(\mathcal I)\cap X_{\mathrm{re}}$, and real quasi-coherent ideal sheaves correspond to the sets $Z\cap X_{\mathrm{re}}$ with $Z\subseteq X$ closed. If $X$ is of finite type over a real closed field $k$, closed real points are precisely the $k$-points. For affine $X$, this recovers the classical real Nullstellensatz $\mathcal I(V_k(I))=\sqrt[\mathrm{re}]{I}$ \cite[Section~4.1]{BCR}. For $M=R$, Theorem~\ref{thm:modfibre} similarly gives $\rad I=\sqrt I$ and $\rrad I$ as the intersection of the real primes containing $I$.
\end{remark}

\begin{remark}[Reduction of a sheaf]\label{rem:reduction}
Let $X$ be Noetherian and $\G\subseteq\F$ quasi-coherent. A local section of a semiprime quasi-coherent sheaf vanishing in every fibre on its domain is zero, since $\rad_X0=\rad0=0$; for real sheaves, real points suffice. Since quasi-coherent subsheaves of semiprime (resp.\ real) sheaves are semiprime (resp.\ real), every morphism $\F\to\mathcal E$ to a semiprime (resp.\ real) quasi-coherent sheaf that kills $\G$ factors uniquely through $\F/\rad\G$ (resp.\ $\F/\rrad\G$). Indeed, its kernel is semiprime (resp.\ real) and contains $\G$, hence also the corresponding radical. For $\F=\OO_X$ and $\G=\mathcal I$, the ordinary quotient is the structure sheaf of the reduced induced closed subscheme on $V(\mathcal I)$ \cite[Tag~01J3]{Stacks}. The factorization property is a sheaf analogue of \cite[Tag~0356]{Stacks}; in general, the correspondence uses realizable families of fibre subspaces in place of closed subsets (Theorem~\ref{thm:sheaffibre}).
\end{remark}

\begin{proof}[Proof of Theorem~\textup{\ref{thmB}(c)}]
Take $S$ to be the set of closed points of $X$, resp.\ of closed points with real residue field, in Theorem~\ref{thm:sheaffibre}(b). Then $\rad\G=\rad_S\G$, resp.\ $\rrad\G=\rad_S\G$, and by Proposition~\ref{prop:primeclass}(a) the sections of $\rad_S\G$ over $U$ are the $s$ with $s(x)\in\G(x)$ for all $x\in S\cap U$.
\end{proof}

\begin{remark}[Semireality for sheaves]\label{rem:sheafsemireal}
Analogously to Lemma~\ref{lem:modproper}, $\mathcal Q=\F/\G$ is \emph{semireal} (has a nonzero real quasi-coherent quotient) if and only if $\rrad\G\neq\F$, if and only if $\mathcal Q(x)\neq0$ for some $x\in X_{\mathrm{re}}$, if and only if a real prime subsheaf contains $\G$. These equivalences follow from Theorem~\ref{thm:sheaffibre} and $\mathcal Q(x)=\F(x)/\G(x)$. For coherent $\mathcal Q$, this means $\operatorname{Supp}\mathcal Q\cap X_{\mathrm{re}}\neq\varnothing$; on schemes of finite type over a field the point can be chosen closed, and rational over a real closed field. Without coherence, $\widetilde{\R(t)}$ on $\mathbb A^1_\R$ has all closed fibres zero (Example~\ref{ex:Q}(e)). For the stalk criterion see Corollary~\ref{cor:sheafsemireallocal}.
\end{remark}

\begin{remark}[Directional $V$--$J$ correspondence]\label{rem:directionalVJ}
Remark~\ref{sec:directional} applies verbatim to sheaves, replacing $(\p,H)$ by $(x,H)$ with $H\subset\F(x)$ a hyperplane and $P_{\p,H}$ by $\mathcal P_{x,H}$; the intersections defining $J$ are quasi-coherent by Lemma~\ref{lem:kernelint}.
\end{remark}

\section{Localization criteria for modules}\label{sec:loc}

We prove compatibility of semiprimeness, reality and their radicals with localization, then give localization criteria for prime and maximal submodules and for semireality.

Let $S\subseteq R$ be a multiplicative set. For a submodule $N\subseteq M$ we regard $S^{-1}N$ as a submodule of the $S^{-1}R$-module $S^{-1}M$, and $\rad(S^{-1}N)$, $\rrad(S^{-1}N)$ are computed over $S^{-1}R$. Recall that $S^{-1}M/S^{-1}N\cong S^{-1}(M/N)$ and that $S^{-1}Q\otimes_{S^{-1}R}S^{-1}Q\cong S^{-1}(Q\otimes_RQ)$ via $(x/s)\otimes(y/t)\mapsto(x\otimes y)/(st)$ \cite[Corollary~3.4(iii) and Proposition~3.7]{AM}.

\begin{lemma}[Clearing denominators]\label{lem:denom}
Let $Q$ be an $R$-module and $S\subseteq R$ a multiplicative set.
\begin{enumerate}
\item[(a)] If $x\in Q$, $s\in S$ and $x/s\in\Ann_{S^{-1}R}(x/s)\,S^{-1}Q$, then $\sigma x\in\Ann_R(\sigma x)\,Q$ for some $\sigma\in S$.
\item[(b)] If $x_1,\dots,x_k\in Q$, $s\in S$ and $\sum_i(x_i/s)\otimes(x_i/s)=0$ in $S^{-1}Q\otimes_{S^{-1}R}S^{-1}Q$, then $\sum_i\sigma x_i\otimes\sigma x_i=0$ in $Q\otimes_RQ$ for some $\sigma\in S$.
\end{enumerate}
\end{lemma}

\begin{proof}
(a) We clear denominators twice: first in an expression of $x/s$ as a linear combination whose coefficients annihilate $x/s$, and then in the annihilation relations themselves. By hypothesis,
\[x/s=\textstyle\sum_{j=1}^l(a_j/t_j)(y_j/u_j)\qquad\text{with }a_j\in R,\ y_j\in Q,\ t_j,u_j\in S\text{ and }(a_j/t_j)(x/s)=0.\]

\emph{The linear combination.} Put $t:=\prod_jt_ju_j\in S$ and $z_j:=\bigl(\prod_{i\neq j}t_iu_i\bigr)y_j\in Q$. Then $x/s=(\sum_ja_jz_j)/t$ in $S^{-1}Q$, so there is $v\in S$ with $v(tx-s\sum_ja_jz_j)=0$, i.e.\ $vtx=\sum_ja_j(vsz_j)$.

\emph{The annihilation relations.} $(a_j/t_j)(x/s)=0$ means $a_jx/(t_js)=0$, so $w_ja_jx=0$ for some $w_j\in S$. Put $w:=\prod_jw_j\in S$.

\emph{Conclusion.} Let $\sigma:=wvt\in S$. Multiplying $vtx=\sum_ja_j(vsz_j)$ by $w$ and using the relations $w_ja_jx=0$, we get
\[\sigma x=\textstyle\sum_j(wa_j)(vsz_j),\qquad (wa_j)\,\sigma x=\sigma\bigl(\prod_{i\neq j}w_i\bigr)(w_ja_jx)=0,\]
so $wa_j\in\Ann_R(\sigma x)$ and $\sigma x\in\Ann_R(\sigma x)Q$.

(b) Under the isomorphism above, the hypothesis says $(\sum_ix_i\otimes x_i)/s^2=0$ in $S^{-1}(Q\otimes Q)$, so $u\sum_ix_i\otimes x_i=0$ for some $u\in S$. Then $\sigma:=u$ works, since $\sum_iux_i\otimes ux_i=u\cdot u\sum_ix_i\otimes x_i$.
\end{proof}

\begin{lemma}[Contraction]\label{lem:contr}
Let $S\subseteq R$ be a multiplicative set and $L$ a semiprime (resp.\ real) $S^{-1}R$-submodule of $S^{-1}M$. Then its contraction $L^c:=\{m\in M:\ m/1\in L\}$ is a semiprime (resp.\ real) submodule of $M$.
\end{lemma}

\begin{proof}
Put $Y:=S^{-1}M/L$, a semiprime (resp.\ real) $S^{-1}R$-module. We use two facts. First, $Y$ is also semiprime (resp.\ real) as an $R$-module: for $x\in Y$ we have $\Ann_R(x)Y\subseteq\Ann_{S^{-1}R}(x)Y$, as the image of $\Ann_R(x)$ in $S^{-1}R$ annihilates $x$; and the natural map $Y\otimes_RY\to Y\otimes_{S^{-1}R}Y$ sends $\sum_ix_i\otimes x_i$ to $\sum_ix_i\otimes x_i$. Second, submodules of semiprime (resp.\ real) modules are semiprime (resp.\ real) by Proposition~\ref{prop:closprop}(a). Now $m+L^c\mapsto m/1+L$ embeds $M/L^c$ into $Y$ by definition of $L^c$, so $M/L^c$ is a semiprime (resp.\ real) module.
\end{proof}

\begin{proposition}[Localization]\label{prop:loc}
Let $N$ be a submodule of an arbitrary $R$-module $M$ and $S\subseteq R$ a multiplicative set.
\begin{enumerate}
\item[(a)] If $M$ is a semiprime (resp.\ real) module, then the $S^{-1}R$-module $S^{-1}M$ is semiprime (resp.\ real).
\item[(b)] If $M_\m$ is a semiprime (resp.\ real) $R_\m$-module for every maximal ideal $\m$, then $M$ is a semiprime (resp.\ real) module.
\item[(c)] $S^{-1}(\rad N)=\rad(S^{-1}N)$ and $S^{-1}(\rrad N)=\rrad(S^{-1}N)$.
\end{enumerate}
Consequently, a submodule $N$ of $M$ is semiprime (resp.\ real) if and only if $N_\p$ is a semiprime (resp.\ real) submodule of $M_\p$ for every prime ideal $\p$, if and only if this holds for every maximal ideal.
\end{proposition}

\begin{proof}
(a) Suppose $M$ is a semiprime module and $x/s\in\Ann_{S^{-1}R}(x/s)S^{-1}M$. Lemma~\ref{lem:denom}(a) gives $\sigma\in S$ with $\sigma x\in\Ann_R(\sigma x)M$, so $\sigma x=0$ and hence $x/s=0$. Suppose $M$ is a real module. Finitely many elements of $S^{-1}M$ can be written with a common denominator, say as $x_i/s$. If $\sum_i(x_i/s)\otimes(x_i/s)=0$, Lemma~\ref{lem:denom}(b) gives $\sigma\in S$ with $\sum_i\sigma x_i\otimes\sigma x_i=0$, so $\sigma x_i=0$ and $x_i/s=0$ for all $i$.

(b) We use twice that if $z\in M$ and $z/1=0$ in $M_\m$ for every maximal ideal $\m$, then $z=0$ \cite[Lemma~2.8(a)]{Eis}. Suppose all $M_\m$ are semiprime and let $q\in\Ann_R(q)M$, say $q=\sum_ja_jy_j$ with $a_jq=0$. For every maximal $\m$ we have $q/1=\sum_j(a_j/1)(y_j/1)$ with $a_j/1\in\Ann_{R_\m}(q/1)$, so $q/1=0$ in $M_\m$. Hence $q=0$. Suppose all $M_\m$ are real and $\sum_iq_i\otimes q_i=0$ in $M\otimes M$. Its image $\sum_i(q_i/1)\otimes(q_i/1)$ in $M_\m\otimes_{R_\m}M_\m\cong(M\otimes M)_\m$ vanishes, so $q_i/1=0$ for all $\m$, and $q_i=0$.

(c) Let $L:=\rad(S^{-1}N)$. By Lemma~\ref{lem:contr}, $L^c$ is a semiprime submodule of $M$ containing $N$, so $\rad N\subseteq L^c$ and hence $S^{-1}(\rad N)\subseteq S^{-1}(L^c)\subseteq L$. Conversely, $S^{-1}M/S^{-1}(\rad N)\cong S^{-1}(M/\rad N)$ is semiprime by (a), as $M/\rad N$ is semiprime, so $S^{-1}(\rad N)$ is a semiprime submodule containing $S^{-1}N$, and therefore $L\subseteq S^{-1}(\rad N)$. The same argument applies to $\rrad$.

The last statement follows by applying (a) and (b) to $M/N$, as $(M/N)_\p=M_\p/N_\p$.
\end{proof}

For prime and maximal submodules, localization alone does not suffice; the support must also be controlled.

\begin{proposition}[Localization criteria for prime and maximal submodules]\label{prop:modlocalclass}
Let $P\subseteq M$, put $Q=M/P$, and write $\operatorname{Supp}Q=\{\p\in\Spec R:Q_\p\neq0\}$.
\begin{enumerate}
\item[(a)] $P$ is prime if and only if $\operatorname{Supp}Q$ has a smallest element $\p$ under inclusion and $P_\q$ is prime in $M_\q$ for every $\q\in\operatorname{Supp}Q$. In this case $\p=(P:M)$ and $\operatorname{Supp}Q=V(\p)$. If $Q$ is finitely generated, the support condition is equivalent to being nonempty and irreducible.
\item[(b)] $P$ is maximal if and only if $\operatorname{Supp}Q=\{\m\}$ for one point $\m$ and $P_\m$ is maximal in $M_\m$. The point $\m$ is necessarily maximal.
\end{enumerate}
\end{proposition}
\begin{proof}
Recall that a nonzero module $Q$ is prime precisely when $ax=0$ implies $x=0$ or $aQ=0$.

(a) Suppose $Q$ is prime and put $\p=\Ann_R Q$. This ideal is prime: if $abQ=0$ but $bQ\neq0$, apply the criterion to a nonzero element $bx$ to get $aQ=0$. The same criterion says that $Q$ is torsion-free over $R/\p$. If $\q\not\supseteq\p$, an element of $\p\setminus\q$ kills $Q$ and becomes invertible in $R_\q$, so $Q_\q=0$. If $\q\supseteq\p$, every element outside $\q$ acts injectively on $Q$, so $Q\to Q_\q$ is injective. Moreover, $Q_\q$ is torsion-free over $R_\q/\p R_\q$: if $(a/s)(x/t)=0$ with $a\notin\p$, then $uax=0$ for some $u\notin\q$, and $ua\notin\p$ forces $x=0$. Thus $Q_\q$ is prime and $\operatorname{Supp}Q=V(\p)$.

Conversely, let $\p$ be the smallest support prime. We first show that $Q\to Q_\p$ is injective. Suppose $0\neq x\in Q$ maps to zero. Choose a maximal ideal $\q$ containing $\Ann_R(x)$; then $x/1\neq0$ in $Q_\q$, so $\q$ belongs to the support and $\p\subseteq\q$. Some $s\notin\p$ satisfies $sx=0$. Primeness of $Q_\q$ gives $sQ_\q=0$. Since $\p\subseteq\q$, inverting all elements outside $\p$ in $Q_\q$ gives $Q_\p$; but $s$ becomes invertible there, forcing $Q_\p=0$, a contradiction.

Now let $ax=0$ with $x\neq0$. Injectivity gives $x/1\neq0$ in the prime module $Q_\p$, so $(a/1)Q_\p=0$. Injectivity then gives $aQ=0$. Hence $Q$ is prime. The forward implication identifies its annihilator and support as claimed.

If $Q$ is finitely generated, its support equals $V(\Ann_R Q)$. Indeed, $Q_\q=0$ means that each of finitely many generators is killed by an element outside $\q$; their product kills $Q$. This closed set has a smallest prime exactly when it is nonempty and irreducible.

(b) If $P$ is maximal, $Q$ is isomorphic to $R/\m$ for a maximal ideal $\m$, giving the stated support and simple localization. Conversely, $\m$ is the smallest support prime and $Q_\m$ is simple, hence prime. By (a), $\m=\Ann_R Q$ and $V(\m)=\{\m\}$, so $\m$ is maximal. Since $\m Q=0$, every element outside $\m$ already acts invertibly on $Q$; thus $Q\to Q_\m$ is an isomorphism. Hence $Q$ is simple and $P$ maximal.
\end{proof}

Semireality has the following localization criterion for arbitrary modules.

\begin{proposition}\label{prop:locsemireal}
Let $M$ be an arbitrary $R$-module and $S\subseteq R$ a multiplicative set.
\begin{enumerate}
\item[(a)] $S^{-1}M$ is a semireal $S^{-1}R$-module if and only if $M$ has a real prime submodule $P$ with $(P:M)\cap S=\varnothing$.
\item[(b)] $M$ is a semireal module if and only if $M_\p$ is semireal for some real prime ideal $\p$, if and only if $M_\m$ is semireal for some maximal ideal $\m$.
\end{enumerate}
\end{proposition}

\begin{proof}
(a) Suppose $P$ is real prime and $(P:M)\cap S=\varnothing$. Put $\p=(P:M)$. The nonzero quotient $M/P$ is torsion-free over $R/\p$, so its localization is nonzero and torsion-free over $S^{-1}(R/\p)$, with annihilator $S^{-1}\p$. Its fraction field is still the real field $\kappa(\p)$. Lemma~\ref{lem:primequot}, applied over $S^{-1}R$, makes $S^{-1}P$ real prime. Thus $S^{-1}M$ is semireal.

Conversely, if $S^{-1}M$ is semireal, choose a real prime submodule $P'$ by \cite[Theorem~B]{Cdec}. Its contraction $P$ is proper, and $M/P$ embeds into the prime module $S^{-1}M/P'$, viewed as an $R$-module. Since nonzero submodules of prime modules are prime, $P$ is prime. Moreover, $(P:M)\cap S=\varnothing$, since the elements of $S$ act invertibly on the nonzero prime quotient. Lemma~\ref{lem:contr} makes $P$ real, proving (a).

(b) If $M$ is semireal, choose a real prime $P$ by \cite[Theorem~B]{Cdec} and put $\p=(P:M)$. For each prime $\q\supseteq\p$, part (a) makes $M_\q$ semireal. Taking $\q=\p$ or a maximal ideal containing $\p$ gives the asserted localizations. Conversely, part (a) shows that semireality of any $M_\q$ implies semireality of $M$.
\end{proof}

The maximal ideal in part (b) need not be real: for $M=\Z$ and $\m=p\Z$, $M_\m=\Z_{(p)}$ is real, although $p\Z$ is not a real ideal.

\section{Stalk criteria for sheaves}\label{sec:loc-sheaves}

Combining the affine criteria of Proposition~\ref{prop:sheafaffine} with module localization gives stalk criteria for semiprimeness, reality and semireality. For prime and maximal subsheaves, the stalk conditions must also specify the support.

\begin{proposition}[Stalk criteria]\label{prop:stalk}
Let $X$ be Noetherian and $\G\subseteq\F$ quasi-coherent.
$\G$ is a semiprime (resp.\ real) subsheaf of $\F$ if and only if $\G_x$ is a semiprime (resp.\ real) submodule of $\F_x$ for every $x\in X$, if and only if this holds for every closed point $x$.
Moreover, for every $x\in X$,
\[
(\rad\G)_x=\rad(\G_x),\qquad
(\rrad\G)_x=\rrad(\G_x),
\]
where the right-hand radicals are taken in $\F_x$ over $\OO_{X,x}$.
\end{proposition}

\begin{proof}
On an affine open $U=\Spec A$ with $M=\Gamma(U,\F)$ and $N=\Gamma(U,\G)$, the stalks at points of $U$ are the localizations $N_\p\subseteq M_\p$. By the last statement of Proposition~\ref{prop:loc}, $N$ is a semiprime (resp.\ real) submodule of $M$ if and only if all $N_\p$ are. With Propositions~\ref{prop:sheafaffine}(a) and~\ref{prop:sheafaffine}(b) this proves the first equivalence. For the second, let $y\in X$. As $X$ is quasi-compact, the closure of $y$ contains a closed point $x$. Every affine open $U\ni x$ also contains $y$, and $(\F/\G)_y$ is a localization of $(\F/\G)_x$; Proposition~\ref{prop:loc}(a) shows that it is semiprime (resp.\ real) if $(\F/\G)_x$ is.

Finally, choose an affine neighbourhood $U=\Spec A$ of $x$, with $M=\Gamma(U,\F)$ and $N=\Gamma(U,\G)$. Formula~\eqref{eq:sheafaffineradicals} identifies the affine restrictions of the sheaf radicals with the module radicals, and Proposition~\ref{prop:loc}(c) identifies their stalks with $\rad(N_\p)$ and $\rrad(N_\p)$. This proves the displayed identities and, together with Proposition~\ref{prop:loc}, Theorem~\ref{thmC}.
\end{proof}

The following proposition gives the sheaf analogue of Proposition~\ref{prop:modlocalclass}(a) for coherent quotients; Corollary~\ref{cor:maxstalk} corresponds to part (b) without a coherence assumption.

\begin{proposition}[Geometric and stalk criteria for primeness]\label{prop:primestalk}
Let $X$ be Noetherian and $\mathcal P\subseteq\F$ quasi-coherent with coherent quotient $\mathcal Q=\F/\mathcal P$. The following are equivalent:
\begin{enumerate}
\item[(a)] $\mathcal P$ is prime.
\item[(b)] $\Ann\mathcal Q=\mathcal I_Z$ for an integral closed subscheme $Z\subseteq X$, and $\mathcal Q$ is torsion-free as an $\OO_Z$-module (on affine opens of $Z$, equivalently at its stalks).
\item[(c)] $\operatorname{Supp}\mathcal Q$ is nonempty and irreducible, and $\mathcal P_x$ is prime in $\F_x$ whenever $\mathcal Q_x\neq0$.
\end{enumerate}
\end{proposition}
\begin{proof}
Coherence gives
\[
(\Ann\mathcal Q)_x=\Ann_{\OO_{X,x}}(\mathcal Q_x),
\qquad \operatorname{Supp}\mathcal Q=V(\Ann\mathcal Q).
\]
These follow on affine opens from finite generation and localization of annihilators.

(a)$\Rightarrow$(b). Let $\eta$ be the point attached to $\mathcal Q$ by Lemma~\ref{lem:primesections}, and give $Z=\overline{\{\eta\}}$ its reduced structure. On every affine $U=\Spec A$ meeting $Z$, the module $\Gamma(U,\mathcal Q)$ is prime with annihilator $\p_\eta$, hence torsion-free over $A/\p_\eta$ by Lemma~\ref{lem:primequot}; on affines disjoint from $Z$ it is zero. These descriptions give (b).

(b)$\Rightarrow$(a). Any nonzero quasi-coherent $\mathcal L\subseteq\mathcal Q$ is coherent and torsion-free over $\OO_Z$. Its generic stalk is nonzero, so its sections on every nonempty affine open of $Z$ are nonzero and torsion-free over the coordinate domain. Thus $\Ann_{\OO_Z}\mathcal L=0$, giving $\mathcal I(\mathcal L)=\Ann_{\OO_X}\mathcal L=\mathcal I_Z=\mathcal I(\mathcal Q)$. The criterion \eqref{eq:sheafcrit} proves primeness.

(b)$\Rightarrow$(c). The support is $Z$, hence nonempty and irreducible. At each $x\in Z$, the stalk $\mathcal Q_x$ is nonzero and torsion-free over the domain $\OO_{Z,x}=\OO_{X,x}/\mathcal I_{Z,x}$, with annihilator $\mathcal I_{Z,x}$. Lemma~\ref{lem:primequot} gives primeness of $\mathcal P_x$.

(c)$\Rightarrow$(b). Put $\mathcal I=\Ann\mathcal Q$. On the support, $\mathcal I_x$ is prime by Lemma~\ref{lem:primequot}; elsewhere it is $\OO_{X,x}$. Thus $\mathcal I$ is radical, and the closed subscheme $Z=V(\mathcal I)$ is reduced with nonempty irreducible underlying space, hence integral. The same lemma makes every stalk of $\mathcal Q$ on $Z$ torsion-free over $\OO_{Z,x}$, proving (b).
\end{proof}

\begin{remark}\label{rem:primesheaf}
Semiprimeness is local (Proposition~\ref{prop:stalk}), but primeness is not: for coherent quotients it additionally requires nonempty irreducible support (Proposition~\ref{prop:primestalk}(c)).

Primeness cannot be checked on stalks alone, or separately on an affine open cover. For $X=\Spec k[t]$, $\F=\OO_X$ and $\mathcal P=\widetilde{(t(t-1))}$, the stalks at $(t)$ and $(t-1)$ are maximal ideals and all other stalks equal $\F_x$. Nevertheless, the quotient corresponds to $k[t]/(t(t-1))\cong k\oplus k$, which is semiprime but not prime. Both restrictions to the affine cover $D(t)\cup D(t-1)$ are prime; the missing condition is irreducibility of the support. The stalk condition in Proposition~\ref{prop:primestalk}(c) is not automatic: for $X=\Spec k[u,v]/(uv)$, the zero subsheaf of $\OO_X$ is semiprime, but its stalk at the origin is not prime because the reduced local ring has zero divisors.
\end{remark}

\begin{corollary}[Stalks of maximal subsheaves]\label{cor:maxstalk}
Let $X$ be Noetherian and $\F$ quasi-coherent. A quasi-coherent $\mathcal M\subseteq\F$ is maximal if and only if there is exactly one point $x$ with $\mathcal M_x\neq\F_x$, and $\mathcal M_x$ is maximal in $\F_x$. This point is necessarily closed.
\end{corollary}
\begin{proof}
A maximal subsheaf has quotient $(j_x)_*\kappa(x)$ for a closed point $x$ by Proposition~\ref{prop:maxsheaf}, so it has exactly one nonzero stalk, which is simple. Conversely, if $\mathcal Q=\F/\mathcal M$ has just one nonzero stalk and it is simple, every quasi-coherent subsheaf $\mathcal L\subseteq\mathcal Q$ has that stalk either zero or equal to $\mathcal Q_x$. All other stalks vanish, so $\mathcal L=0$ or $\mathcal Q$. Hence $\mathcal M$ is maximal. The point $x$ is closed also directly: for every specialization $y$ of $x$, $\mathcal Q_x$ is a localization of $\mathcal Q_y$, so $\mathcal Q_y\neq0$ and therefore $y=x$.
\end{proof}

\begin{corollary}[Stalk criterion for semireality]\label{cor:sheafsemireallocal}
Let $X$ be Noetherian and $\mathcal Q$ quasi-coherent. The following are equivalent:
\begin{enumerate}
\item[(a)] $\mathcal Q$ is semireal, i.e.\ it has a nonzero real quasi-coherent quotient.
\item[(b)] $\Gamma(U,\mathcal Q)$ is a semireal $\Gamma(U,\OO_X)$-module for some affine open $U\subseteq X$.
\item[(c)] $\mathcal Q_x$ is a semireal $\OO_{X,x}$-module for some $x\in X$.
\end{enumerate}
No coherence assumption is required.
\end{corollary}
\begin{proof}
By \eqref{eq:sheafaffineradicals}, the real radical of $0\subseteq\mathcal Q$ restricts on $U=\Spec A$ to the real radical of $0\subseteq\Gamma(U,\mathcal Q)$. Thus $\rrad0\neq\mathcal Q$ if and only if this inequality holds on some affine open, proving (a)$\Leftrightarrow$(b). If (b) holds, Proposition~\ref{prop:locsemireal}(b) gives a semireal localization, hence a semireal stalk. Conversely, for an affine neighbourhood $U=\Spec A$ of $x\leftrightarrow\p$, we have $\mathcal Q_x=\Gamma(U,\mathcal Q)_\p$. Proposition~\ref{prop:locsemireal}(a) shows that semireality of this localization implies semireality of $\Gamma(U,\mathcal Q)$.
\end{proof}

The residue field at the detecting stalk need not be real: the stalk $\Z_{(p)}$ of $\OO_{\Spec\Z}$ is semireal, although its residue field $\mathbb F_p$ is not real. Semireality is detected at some stalk, whereas reality must hold at every stalk (Proposition~\ref{prop:stalk}).

\section{A projective example}\label{sec:example}

On an affine scheme, the criterion \eqref{eq:sheafcrit} and Definition~\ref{def:realsheaf} recover semiprime and real submodules by Proposition~\ref{prop:sheafaffine}. Theorem~\ref{thm:sheaffibre}(a),(b) expresses the radicals as intersections of prime and real prime subsheaves, without a finiteness assumption on $\F$. For coherent $\F$ on a finite-type scheme, Theorem~\ref{thm:sheaffibre}(b) gives the maximal and real maximal intersections and the corresponding closed-point tests.

The following example compares the ordinary and real radicals on a projective scheme and on its standard affine cover, and shows why they cannot be computed from global sections alone.

\begin{example}[Multiplication by $x_0+ix_1$ on $\mathbb P^2_\R$]\label{ex:projective}
Let $X=\mathbb P^2_\R$, $\F=\OO_X(2)^2$, $Z=V(x_0^2+x_1^2)$ and $p=[0:0:1]$. Define
\[
\Theta=\begin{pmatrix}x_0&-x_1\\x_1&x_0\end{pmatrix},\qquad
\begin{aligned}
\mathcal H&=\operatorname{im}(\Theta:\OO_X(1)^2\to\F),\\
\G&=\operatorname{im}(\Theta^2:\OO_X^2\to\F).
\end{aligned}
\]
The adjugate identities and nonzero determinants give injectivity, so $\mathcal H\cong\OO_X(1)^2$ and $\G\cong\OO_X^2$. We shall prove
\[
\G\subsetneq\rad\G=\mathcal H\subsetneq\rrad\G=\mathcal I_p\F\subsetneq\F.
\]

\emph{Geometry and affine description.} The conic $Z$ is integral over $\R$, with unique singular and real point $p$. Over $\mathbb C$ it splits into the conjugate lines $V(x_0\pm ix_1)$, meeting at $p$. Its complement of $p$ is
\[
Z\setminus\{p\}=\Spec\R[s,t]/(1+s^2)\cong\Spec\mathbb C[t],
\qquad (s,t)=(x_1/x_0,x_2/x_0).
\]
Thus its other closed points have residue field $\mathbb C$, and its generic point $\eta$ has residue field $\mathbb C(t)$.

On $U_j=D_+(x_j)=\Spec A_j$, trivialize $\OO_X(k)$ by $x_j^k$. Put
\[
\begin{aligned}
A_0&=\R[s,t], &(s,t)&=(x_1/x_0,x_2/x_0), &w_0&=1+is,\\
A_1&=\R[r,h], &(r,h)&=(x_0/x_1,x_2/x_1), &w_1&=r+i,\\
A_2&=\R[u,v], &(u,v)&=(x_0/x_2,x_1/x_2), &w_2&=u+iv.
\end{aligned}
\]
Identify $A_j^2$ with $B_j:=A_j\otimes_\R\mathbb C$ via $(a,b)\mapsto a+ib$. The maps become multiplication by $w_j$ and $w_j^2$, so the sections of $\mathcal H$ and $\G$ are $w_jB_j$ and $w_j^2B_j$. Moreover,
\[
B_0/(w_0)\cong\mathbb C[t],\qquad
B_1/(w_1)\cong\mathbb C[h],\qquad
B_2/(w_2)\cong\mathbb C[v],
\]
with $s\mapsto i$, $r\mapsto-i$ and $u\mapsto-iv$. Their annihilators over $A_j$ are $(1+s^2)$, $(r^2+1)$ and $(u^2+v^2)$: in the last case, divisibility of a real polynomial by $u+iv$ also implies divisibility by $u-iv$, hence by their product. These quotients are torsion-free over the coordinate domains of $Z$. By Proposition~\ref{prop:primestalk}(b), $\mathcal H$ is prime, hence semiprime.

\emph{Radicals and fibre images.} At any point $x\in U_j$, the fibre matrix is
\[
B=\begin{pmatrix}a&-b\\b&a\end{pmatrix},\qquad
(a,b)=(x_0/x_j,x_1/x_j)(x).
\]
It is invertible outside $Z$ and zero at $p$. On $Z\setminus\{p\}$, $a^2+b^2=0$ and $a\neq0$, so $B$ has rank $1$ and $B^2=2aB$. Hence $\G(x)=\mathcal H(x)$ at every point. Since $\mathcal H$ is semiprime, Theorem~\ref{thm:sheaffibre}(b) gives
\[
\rad\G=\rad\mathcal H=\mathcal H.
\]
Among real closed points, only $p$ has $\G(x)\neq\F(x)$, and $\G(p)=0$. Theorem~\ref{thm:sheaffibre}(b) therefore gives $\rrad\G=\mathcal I_p\F$. The resulting dimensions of images in $\F(x)$, over $\kappa(x)$, are:
\begin{center}\small\renewcommand{\arraystretch}{1.3}
\begin{tabular}{@{}lcccc@{}}
point $x$ & $\G(x)$ & $(\rad\G)(x)$ & $(\rrad\G)(x)$ & $\F(x)$ \\ \hline
$x\notin Z$ & $2$ & $2$ & $2$ & $2$ \\
$x\in Z\setminus\{p\}$ & $1$ & $1$ & $2$ & $2$ \\
$x=p$ & $0$ & $0$ & $0$ & $2$
\end{tabular}
\end{center}
The last two rows, together with $w_j^2B_j\subsetneq w_jB_j$, prove the strict inclusions above. Equality of the first two columns illustrates Theorem~\ref{thm:sheaffibre}: a subsheaf and its semiprime radical have the same images in every fibre.

These images need not be the intrinsic fibres $\mathcal E_x/\m_x\mathcal E_x$. At $p$, the intrinsic fibres of $\G$ and $\mathcal H$ have dimension $2$, whereas that of $\mathcal I_p\F$ has dimension $4$, since $(\mathcal I_p\F)_p=\m_p^{\oplus2}$ and $\m_p=(u,v)\OO_{X,p}$. Thus $\mathcal H_p\subsetneq(\mathcal I_p\F)_p$ although both images in $\F(p)$ vanish.

\emph{Sections.} Let $S_k$ denote the real homogeneous forms of degree $k$ and put $w=x_0+ix_1$, so $w_j=w/x_j$. Identifying pairs of forms with $S_k\otimes_\R\mathbb C$, we obtain the following table; the last column includes dimensions over $\R$.
\begin{center}\small\renewcommand{\arraystretch}{1.3}
\begin{tabular}{@{}lcccc@{}}
sections over & $U_0$ & $U_1$ & $U_2$ & $X$ \\ \hline
$\G$ & $w_0^2B_0$ & $w_1^2B_1$ & $w_2^2B_2$ & $w^2\mathbb C$ \ ($2$) \\
$\rad\G$ & $w_0B_0$ & $w_1B_1$ & $w_2B_2$ & $w(S_1\otimes\mathbb C)$ \ ($6$) \\
$\rrad\G$ & $B_0$ & $B_1$ & $(u,v)B_2$ & $\{f\in S_2\otimes\mathbb C:\ f(p)=0\}$ \ ($10$) \\
$\F$ & $B_0$ & $B_1$ & $B_2$ & $S_2\otimes\mathbb C$ \ ($12$)
\end{tabular}
\end{center}
The global dimensions follow from $\G\cong\OO_X^2$, $\mathcal H\cong\OO_X(1)^2$, and $\dim S_2=6$, with vanishing at $p$ imposing one condition on each component. Since $\Gamma(X,\OO_X)=\R$ is a real field, every subspace of $\Gamma(X,\F)$ is a real, hence semiprime, $\R$-submodule, so both module radicals of $\Gamma(X,\G)$ equal $\Gamma(X,\G)$ itself; they do not recover the sections of the sheaf radicals, of dimensions $6$ and $10$.

For every open $U\subseteq X$, the supports of the quotients give
\[
(\rad\G)|_U\neq\F|_U\iff U\cap Z\neq\varnothing,\qquad
(\rrad\G)|_U\neq\F|_U\iff p\in U.
\]
In particular, removing $p$ makes the real radical the whole sheaf, while the ordinary radical remains proper.

\emph{Multiplicities.} Fibre images do not distinguish $\G$ from $\mathcal H$, but quotient lengths do. At $\eta$, the local ring $D=\OO_{X,\eta}=(A_2)_{(\pi)}$ is a discrete valuation ring with uniformizer $\pi=u^2+v^2$. The matrices $\Theta$ and $\Theta^2$ have unit entries $u$ and $u^2-v^2\equiv2u^2\pmod\pi$, and determinants $\pi$ and $\pi^2$. Row and column operations over $D$ therefore give
\[
(\F/\G)_\eta\cong D/(\pi^2),\qquad
(\F/\mathcal H)_\eta\cong D/(\pi).
\]
The lengths over the respective local rings are:
\begin{center}\small\renewcommand{\arraystretch}{1.3}
\begin{tabular}{@{}lccc@{}}
quotient & support & length at $\eta$ & length at $p$ \\ \hline
$\F/\G$ & $Z$ & $2$ & $\infty$ \\
$\F/\rad\G$ & $Z$ & $1$ & $\infty$ \\
$\F/\rrad\G$ & $\{p\}$ & $0$ & $2$
\end{tabular}
\end{center}
Here $\infty$ means no finite composition series: the first two quotients have one-dimensional support at $p$, while the third is the skyscraper $\R^2$. The ordinary radical removes multiplicity along $Z$; the real radical leaves a quotient supported only at $p$.

\emph{Direct witnesses.} The failures of semiprimeness of $\G$ and reality of $\mathcal H$ also follow from the intrinsic definitions. For $\mathcal Q=\F/\G$, the adjugate identity gives $\mathcal I_Z^2\mathcal Q=0$, but $\mathcal L=\mathcal I_Z\mathcal Q\neq0$, since $u^2+v^2\notin(u+iv)^2B_2$. Thus $\mathcal I_Z\mathcal L=0$ and $\mathcal L\subseteq\mathcal I_Z\mathcal Q$, violating \eqref{eq:sheafidealcrit}. For $\mathcal T=\F/\mathcal H$, the image of $(x_0,x_1):\mathcal T(-1)\to\mathcal T^{\oplus2}$ is nonzero and totally isotropic, as its tensor pairing factors through multiplication by $x_0^2+x_1^2$, which kills $\mathcal T$; so $\mathcal H$ is not real by Definition~\ref{def:realsheaf}.

\emph{Prime and maximal intersections.} Since $\G(\eta)=\mathcal H(\eta)$, $\F/\mathcal H$ is torsion-free on $Z$, and $(\F/\mathcal H)_\eta$ is killed by $\m_\eta$, and since $\mathcal I_p\F$ consists of the sections vanishing at $p$, we have
\[
\mathcal H=\mathcal P_{\eta,\G(\eta)},\qquad
\mathcal I_p\F=\mathcal P_{p,0}.
\]
These are the smallest prime and real prime subsheaves containing $\G$. The maximal subsheaves over $\G$ are $\mathcal P_{p,L}$ for lines $L\subset\R^2$, and $\mathcal P_{x,\G(x)}$ for non-real closed $x\in Z$. Intersecting all of them gives $\mathcal H$; intersecting just the real ones gives $\mathcal I_p\F$.

Finally, for $\F_\infty=\bigoplus_{n\ge1}\F$ and $\G_\infty=\bigoplus_{n\ge1}\G$, the direct-sum formulas \eqref{eq:radcoproduct} on affine opens give
\[
\rad\G_\infty=\bigoplus_{n\ge1}\mathcal H,\qquad
\rrad\G_\infty=\mathcal I_p\F_\infty.
\]
This retains the chart distinctions and gives a real prime quotient at $p$ with infinite-dimensional fibre, illustrating the intersection results for noncoherent quasi-coherent sheaves.
\end{example}

\section{Questions}\label{sec:questions}

\begin{question}
Which modules over a Jacobson ring are Jacobson? By the linear separation argument in the proof of Theorem~\ref{thm:jac} (cf.\ \cite[Lemma~2.6]{MSS}), $M$ is Jacobson if and only if $\bigcap_\m\m(M/P)=0$ for every prime submodule $P$ of $M$. By Theorem~\ref{thm:jac} it suffices that every $M/P$ be torsionless over $R/(P:M)$, and by Example~\ref{ex:Q}(d) this is not necessary; by Example~\ref{ex:Q}(a), without finite generation $M$ need not be Jacobson. Is there a useful intrinsic description, e.g.\ over $k[x,y]$?
\end{question}

\begin{question}\label{q:effective}
The radicals admit iterative descriptions, which suggest an effective version of the Nullstellens\"atze. For a submodule $N$ of an $R$-module $M$, let $N^{(1)}$ be generated by all $m\in M$ with $m\in(N:m)M$, and let $N^{[1]}$ be generated by all $m_i$ occurring in finite families with $\sum_i\bar m_i\otimes\bar m_i=0$ in $(M/N)\otimes_R(M/N)$. Starting with $N^{(0)}=N^{[0]}=N$, iterate these operations:
\[
N^{(k+1)}=(N^{(k)})^{(1)},\qquad
N^{[k+1]}=(N^{[k]})^{[1]}.
\]
For arbitrary modules,
\[
\rad N=\bigcup_{k\ge0}N^{(k)},\qquad
\rrad N=\bigcup_{k\ge0}N^{[k]}.
\]
The ordinary formula for finitely generated $M$ is \cite[Theorem~2.7]{Ary}. For $M=R$, the operations adjoin the $r$ with $r^2\in N$, respectively the $r_i$ with $\sum_i r_i^2\in N$.

For finitely generated modules, Marcelo and Rodr\'iguez characterize the prime radical as the degree-one part of an ideal radical in the symmetric algebra \cite[Corollary~2.4 and Proposition~3.2]{MR}; Section~3.2 of their paper uses this to compute generators for submodules of polynomial free modules. The question here concerns bounds on the iterative constructions above.

For finitely generated modules over Noetherian rings both chains stabilize, but the number of steps is not bounded in terms of $R$ and $M$ alone: for $R=M=\R[x]$ and $N=(x^{2^h})$ one has
\[
N^{(j)}=N^{[j]}=(x^{2^{h-j}})\qquad(0\le j\le h),
\]
so reaching $\rad N=\rrad N=(x)$ takes exactly $h$ steps. For $R=k[x_1,\dots,x_n]$ and $M=R^m$, are there effective bounds on the number of steps in terms of $n$, $m$, the number of generators of $N$ and their degrees, in the spirit of the effective Nullstellensatz?
\end{question}

\section*{Acknowledgements}
In the preparation of this paper the author used the AI assistants Claude (Anthropic) and ChatGPT (OpenAI) for conceptual discussions regarding proofs and literature, and for editing assistance. All arguments and references were independently verified by the author, who takes full responsibility for the content.

\enlargethispage{2\baselineskip}

\end{document}